\documentclass[a4paper,11pt]{amsart}
\usepackage[utf8]{inputenc}
\usepackage[a4paper]{geometry} 
\usepackage{mathtools,amssymb,bm,amsthm} 
\usepackage{amsfonts}
\usepackage[utf8]{inputenc}
\usepackage[english]{babel} 
\usepackage{xcolor} 
\usepackage{graphicx} 
\usepackage[pdfencoding=auto]{hyperref} 
\usepackage{stmaryrd, enumitem}
\usepackage{dsfont}
\newtheorem{theorem}{Theorem}[section]

\newtheorem{corollary}[theorem]{Corollary}
\newtheorem{lemma}[theorem]{Lemma}
\newtheorem{fact}[theorem]{Fact}
\newtheorem*{step1}{Step 1}
\newtheorem*{step2}{Step 2}
\newtheorem*{step3}{Step 3}
\newtheorem*{step4}{Step 4}
\newtheorem*{step5}{Step 5}
\newtheorem*{step6}{Step 6}

\theoremstyle{remark}
\newtheorem*{remark}{Remark}

\theoremstyle{definition}

\newtheorem{Remark}[theorem]{Remark}
\numberwithin{equation}{section}
\newcommand{\R}{\mathbb{R}}
\newcommand{\N}{\mathbb{N}}
\newcommand{\Z}{\mathbb{Z}}

\newcommand{\ind}{\mathds{1}}

\newcommand{\minus}{\setminus}
\newcommand{\comma}{\mathpunct{\raisebox{0.5ex}{,}}}

\author[M. Fakhoury]{Micheline Fakhoury}
\address{Univ. Artois, UR 2462, Laboratoire de Math\'{e}matiques de Lens (LML)\\ F-62300 Lens, France}
\email{micheline.fakhoury@univ-artois.fr}
\begin{document}

\title[Plasticity of the unit ball of $C(K)$]{Plasticity of the unit ball \\of some $C(K)$ spaces}

\begin{abstract} We show that if $K$ is a compact metrizable space with finitely many accumulation points, then the closed unit ball of $C(K)$ is a plastic metric space, which means that any non-expansive bijection from $B_{C(K)}$ onto itself is in fact an isometry. We also show that if $K$ is a zero-dimensional compact Hausdorff space with a dense set of isolated points, then any non-expansive homeomorphism of $B_{C(K)}$ is an isometry.
\end{abstract}

\thanks{The author would like to thank \'Etienne Matheron for his help regarding this paper.}
\thanks{This work was supported in part by
	the project FRONT of the French
	National Research Agency (grant ANR-17-CE40-0021).}
\maketitle
\section{Introduction}

Following \cite{NPW}, a metric space $(M,\rho)$ is said to be \emph{plastic} if every non-expansive bijection $F:M\to M$ is in fact an isometry (``non-expansive'' means ``$1$-Lipschitz''). For example, it is a well known classical fact that any compact metric space is plastic; more generally, any totally bounded metric space is plastic (see \cite{NPW} for a proof and for historical references, and see \cite{AKZ} for an extension to the setting of uniform spaces). On the other hand, every non-trivial normed space is non-plastic, as shown by $F(x)=\frac12 x$. Several interesting examples can be found in \cite{NPW}.

\medskip
In recent years, the following intriguing problem was considered by a number of authors: is it true that for any real Banach space $X$, the closed unit ball $B_X$ is plastic? As the formulation suggests, no counterexample is known. On the other hand, several positive results have been obtained. For example, the following Banach spaces have a plastic unit ball (we state the results more or less in chronological order):
\begin{itemize}
\item[-] Any finite-dimensional normed space $X$, since in this case $B_X$ is compact.
\item[-] Any strictly convex space (\cite{CKOW}). 
\item[-] The space $\ell_1$ (\cite{KZ1}).
\item[-] Any Banach space whose unit sphere is the union of its finite-dimensional polyhedral extreme subsets (\cite{AKZ}).
\item[-] Any $\ell_1\,$-$\,$direct sum of strictly convex spaces (\cite{KZ2}).
\item[-] The space $c$ of all convergent sequences of real numbers (\cite{L}); and the space $c_0$ if one weakens the definition of plasticity by considering only non-expansive bijections with a continuous inverse (\cite{L}).
\item[-] Any $\ell_\infty\,$-$\,$direct sum of  \emph{two} strictly convex spaces (\cite{HLZ}).
\item[-] The space $\ell_1\oplus_2 \R$ (\cite{HLZ}).
\end{itemize}

\medskip Obvious ``missing examples'' are the spaces $L_1=L_1(0,1)$ or $C(K)$, where $K$ is a compact Hausdorff space (here, of course, $C(K)$ is the Banach space of all real-valued continuous functions on $K$). 

\smallskip
Note that the space $c$ mentioned above is a $C(K)$ space, in fact the simplest infinite-dimensional such space: indeed, $c$ is isometric to $C(K)$ for any compact metric space $K$ with exactly $1$ accumulation point. In this paper, we enlarge a little bit the list of Banach spaces with a plastic unit ball by proving the following result.

\begin{theorem}\label{main} If $K$ is a  compact metrizable space with finitely many accumulation points, then $B_{C(K)}$ is plastic.
\end{theorem}

\smallskip
The assumptions made on the space $K$ in Theorem~\ref{main} are, of course, extremely strong. However, it will turn out that in most parts of the proof they can be relaxed; and this will give the following result, whose conclusion is weaker but which applies to a much more general class of compact Hausdorff spaces.  Recall that a topological space is said to be \emph{zero-dimensional} if it has a basis consisting of clopen sets.
\begin{theorem}\label{main2} If $K$ is a zero-dimensional compact Hausdorff space with a dense set of isolated points, then any non-expansive homeomorphism $F:B_{C(K)}\to B_{C(K)}$ is an isometry.
\end{theorem} 

\begin{corollary} If $K$ is a countable compact Hausdorff space, then any non-expansive homeomorphism of $B_{C(K)}$ is an isometry.
\end{corollary}
\begin{proof} This is obvious by Theorem~\ref{main2}.
\end{proof}

\begin{corollary} Any non-expansive homeomorphism of $B_{\ell_\infty}$ is an isometry.
\end{corollary}
\begin{proof} The space $\ell_\infty$ is isometric to $C(\beta\N)$, where $\beta\N$ is the Stone-$\check{\rm C}$ech compactification of $\N$. The space $\beta\N$ is zero-dimensional with a dense set of isolated points (namely $\N$), so we may apply Theorem~\ref{main2}.
\end{proof}

It would be nice to remove the assumption that $F^{-1}$ is continuous in Theorem~\ref{main2}, since this would say that $B_{C(K)}$ is plastic. Actually, it seems  quite plausible that for an arbitrary Banach space $X$, any non-expansive bijection $F:B_X\to B_X$ is a homeomorphism; but we have been unable to prove that.

\medskip The paper is organized as follows. In Section~\ref{S2}, we collect some preliminary facts. In Section~\ref{S3}, which is the longest part of the paper, we prove Theorem~\ref{main}. As may be guessed, several ideas are borrowed from the clever proof given by N. Leo \cite{L} in the case of $c$. However, exactly as in the cases of $\ell_1\,$-$\,$sums of strictly convex spaces in \cite{KZ2} and of $\ell_\infty\,$-$\,$sums of two strictly convex spaces in \cite{HLZ}, a number of technicalities appear, which makes the whole proof rather long. (What is in fact missing is a general result concerning direct sums of Banach spaces with plastic unit balls, which is yet to be found.) In Section~\ref{S4}, we prove Theorem~\ref{main2} by slightly modifying the proof of Theorem~\ref{main}. Finally, in Section~\ref{S5} we prove some additional results in the spirit of Theorem~\ref{main2}.

\medskip 

\section{Preliminary facts}\label{S2}

In this short section, we collect a few results that will be needed for the proof of Theorem~\ref{main}.

\medskip
The following general theorem was proved in \cite{CKOW}.
\begin{theorem}\label{th1}
Let $X$ be a normed space. Denote by $S_X$ and $B_X$ the unit sphere and the closed unit ball of $X$ respectively, and by $\mathrm{ext} (B_X)$ the set of extreme points of $B_X$. 
 The following facts hold true for any non-expansive bijection $F:B_X \rightarrow B_X$.
\begin{enumerate}[label={\rm \arabic*)}]
\item $F(0)=0$,
\item if $x \in S_X $, then $F^{-1}(x) \in S_X$,
\item if $x \in \mathrm{ext }(B_X)$, then $F^{-1}(x) \in\mathrm{ext }(B_X)$ and $F^{-1}(\alpha x)= \alpha F^{-1}(x)$ for every $\qquad \alpha \in [-1,1]$.
\end{enumerate}
\end{theorem}

\smallskip Next, we state Mankiewicz's local version of the classical Mazur-Ulam Theorem (\cite{M}, see also \cite[Theorem 14.1]{BL}).
\begin{theorem}\label{th2}
Let $X$ and $Y$ be normed spaces and let $U$ be a subset of $X$ and $V$ be a subset of $Y$. If $U$ and $V$ are convex with non-empty interior, then any isometric bijection $\Psi : U \rightarrow V $ extends to an affine isometric bijection $\widetilde \Psi : X \rightarrow Y$. 
\end{theorem}

\smallskip The following lemma is the natural generalization of \cite[Lemma 4.1]{L}. The proof is very similar.
\begin{lemma}\label{prop1}
Let $K$ be a compact Hausdorff space, and let $f$ and $g$ be two non-zero elements of $B_{C(K)}$. 
Then, the closed balls $\overline B(f,1)$ and  $\overline B(g,1)$ cover $B_{C(K)}$ if, and only if, there exist an isolated point $u$ of $K$ and real numbers $c,d$ such that $f= c \mathds{1}_{\{u\}}$, $g= d\mathds{1}_{\{u\}}$ and $cd <0$.
\end{lemma}
\begin{proof}
Suppose that $B_{C(K)} \subset \overline B(f,1) \cup \overline B (g,1)$.\\
Since $f \neq 0$, we may choose $u \in K$ such that $f(u) \neq 0$.\\
Suppose that there exists $v \neq u \in K$ such that $g(v) \neq 0$. By the Tietze extension Theorem, one can find a function $h\in B_{C(K)}$ such that $h(u)=- \mathrm{sgn}\bigl(f(u)\bigr)$ and $h(v)=- \mathrm{sgn}\bigl(g(v)\bigr)$. Then, $\vert h(u)-f(u)\vert >1$ and $\vert h(v)-g(v)\vert >1$; so $h\notin \overline B(f,1)\cup \overline B(g,1)$, which is a contradiction.\\ 
Hence, we have $g(x)=0$, for all $x\neq u$. Since $g\neq 0$, it follows that   $g=d \mathds{1}_{\{u\}}$ for some real number $d \neq 0$. Note that this  implies  that $u$ is an isolated point of $K$ because $g$ is continuous.\\
Now, if we repeat the same argument, we get $f(x) =0$, for all $ x \neq u$.  
Hence, $f=c\mathds{1}_{\{u\}}$ for some real number $c \neq 0$. \\
Finally we have $cd<0$. Indeed, 
if $cd>0$, then the function $- \mathrm{sgn}(c)\mathds{1}_{\{u\}}=- \mathrm{sgn}(d)\mathds{1}_{\{u\}}\in B_{C(K)}$ does not belong to $\overline B (f,1) \cup \overline B (g,1)$, which is a contradiction. 

\medskip
For the converse implication, suppose that there exists an isolated point $u$ of $K$ such that $f= c \mathds{1}_{\{u\}}$, $g= d\mathds{1}_{\{u\}}$ and $cd <0$.\\
Let $h \in B_{C(K)}$ be arbitrary. If $h(u)=0$, then $h \in \overline B(f,1) \cap \overline B(g,1)$. 
If $h(u) \neq 0$, then either $\mathrm{sgn}\bigl(h(u)\bigr)=\mathrm{sgn}(c)$, in which case $h \in \overline B(f,1)$, or $\mathrm{sgn}\bigl(h(u)\bigr)=\mathrm{sgn}(d)$, in which case $h \in \overline B(g,1)$. Hence, we see that $B_{C(K)} \subset \overline B(f,1)\cup \overline B(g,1)$. 
\end{proof}

\medskip
Finally, we will use a few times the following  well known fact. 
\begin{Remark}\label{dense} Let $K$ be a zero-dimensional compact Hausdorff space. If $A\subset\R$ has empty interior, then the set $D=\{ f\in C(K);\; \forall x\in K,\; f(x)\notin A\}$ is dense in $C(K)$.
\end{Remark}
\begin{proof} Since $K$ is zero-dimensional, the set of all $\varphi\in C(K)$ taking only finitely many values is dense in $C(K)$ by the Stone-Weierstrass Theorem. So it is enough to show that any such function $\varphi$ can be approximated by some $f\in D$; which is clear since $\R\setminus A$ is dense in $\R$.
\end{proof}
\section{Proof of Theorem \ref{main}}\label{S3}

In what follows, we consider a compact Hausdorff space $K$ with a finite number of accumulation points. We denote by $K'$ the set of all accumulation points of $K$. We also assume that $K$ is infinite (so $K'\neq\emptyset$), since otherwise $C(K)$ is finite-dimensional and we already know that $B_{C(K)}$ is plastic.

\smallskip Note that $K$ is necessarily zero-dimensional. Moreover, $K\setminus K'$ is dense in $K$, a fact that will be used repeatedly in the proof of Theorem~\ref{main}.

\smallskip For simplicity, we write $B$ instead of $B_{C(K)}$, and we denote by $\overline B(f,r)$ the closed ball with center $f\in C(K)$ and radius $r$.

\smallskip
From now on, we fix a non-expansive bijection $F:B\to B$. Our aim is to show that $F$ is an isometry. At some point we will need to assume that $K$ is metrizable; but for the beginning of the proof this is not necessary.

\subsection{A reduction} The following lemma will allow us to reduce the proof that $F$ is an isometry to a more tractable job, namely proving that a suitably defined non-expansive bijection from $B$ onto itself with some  additional properties is the identity map.
\begin{lemma}\label{l1}
There exist a  homeomorphism $\sigma : K \rightarrow K$ and a continuous function $\alpha :K  \rightarrow  \{-1,1\}$ such that, for every $f \in B$ and $a \in K \setminus K'$, we have the following:
\begin{enumerate}[label={\rm \arabic*)}]
\item if $f(a)=0$, then $F(f)(\sigma(a))=0$, 
\item if $f(a) \alpha (a)<0$, then $F(f)(\sigma(a))\leqslant 0$, 
\item if $f(a) \alpha(a)>0$, then $F(f)(\sigma(a))\geqslant 0$.
\end{enumerate}
\end{lemma}
\begin{proof} As suggested above, in this proof we will not assume that the space $K$ is metrizable; so we will use nets rather than sequences.

\smallskip 
We will first define a map $\sigma_0:K\setminus K'\rightarrow K \minus K'$ and a continuous function $\alpha:K \to \{-1,1\}$, then show that $\sigma_0$ can be extended to a homeomorphism $\sigma :K\to K$, and finally prove 1), 2) and 3).

\begin{step1} For any $a\in K\setminus K'$, there exist a point $\sigma_0(a)\in K\setminus K'$ and a continuous and strictly monotonic function $c_a:[-1,1]\to [-1,1]$ with $c_a(0)=0$ $($so $c_a(-1)c_a(1)<0)$ such that
\[F(t\ind_{\{ a\}})=c_a(t)\ind_{\{ \sigma_0(a)\}} \qquad\hbox{for all}\;\; t \in [-1,1].\]
\end{step1}
\begin{proof} Let $a \in K \minus K'$.\\ Lemma~\ref{prop1} implies that $B \subset \overline B (\ind_{\{a\}},1) \cup \overline B (-\ind_{\{a\}},1)$. As $F$ is non-expansive and surjective, it follows that $B \subset \overline B \bigl(F(\ind_{\{a\}}),1\bigr) \cup \overline B \bigl(F(-\ind_{\{a\}}),1\bigr)$. Moreover, 
Theorem~\ref{th1} (item~1) implies that $F(\ind_{\{a\}})$ and $F(-\ind_{\{a\}})$ are non-zero. By Lemma~\ref{prop1}, one can find $\sigma_0(a)\in K\setminus K'$ such that  $F(\ind_{\{ a\}})$ and $F(-\ind_{\{ a\}})$ are non-zero multiples of $\ind_{\{ \sigma_0(a)\}}$ and $F(\ind_{\{ a\}})(\sigma_0(a))\, F(- \ind_{\{ a\}})(\sigma_0(a))<0$.

\smallskip\noindent
Let us show that for any $t\in [-1,1]$, we have $F(t\ind_{\{ a\}})=c_a(t)\ind_{\{ \sigma_0(a)\}}$ for some real number $c_a(t)\in [-1,1]$. \\Let $t \in [-1,1]$ and let $f=t \ind_{\{a\}}$. If $t=0$, then $f=0$, and Theorem~\ref{th1} implies that $F(f)=0$.
If $t >0$, then Lemma~\ref{prop1} implies that $B \subset \overline B(f,1) \cup \overline B (-\ind _{\{a\}},1)$. As $F$ is non-expansive and surjective, it follows that $B \subset \overline B\bigl(F(f),1\bigr) \cup \overline B \bigl(F(-\ind _{\{a\}}),1\bigr)$. Since $F(f)$ and $F(-\ind _{\{a\}})$ are non-zero by Theorem~\ref{th1} (item~1), Lemma~\ref{prop1} and  the definition of $\sigma_0(a)$ imply that $F(f)$ is a non-zero multiple of 
$\ind_{\{ \sigma_0(a)\}}$. The case $t < 0$ is analogous. So we have shown that $F(t\ind_{\{a\}})=c_a(t) \ind_{\{\sigma_0(a)\}}$ for some real number $c_a(t)$, and $c_a(t)\in [-1,1]$ since $F(t\ind_{\{a\}})\in B$.\\
We have $c_a(t)= F(t\ind_{\{a\}})(\sigma_0(a))$ for all $t\in [-1,1]$. Since $F$ is continuous, this shows that the function $c_a:[-1,1]\to [-1,1]$ is continuous. Moreover, $c_a$ is also injective because $F$ is injective. Hence, $c_a$ is  strictly monotonic. 
\end{proof}

\smallskip Let us now define the continuous function $\alpha:= F^{-1}(\ind)$, where $\ind$ is the constant function.\\
Since $\ind$ is an extreme point of $B$, Theorem~\ref{th1} (item~3) implies that $\alpha$ is also an extreme point of $B$, so $\vert\alpha\vert\equiv 1$. \\

\smallskip The following fact will be used several times in the remaining of the proof.

\begin{fact}\label{keyformula} If $g$ is any extreme point of $B$ and $f=F^{-1}(g)$, then 
\[ g(\sigma_0(a))=\alpha(a)f(a)\qquad\hbox{for all}\; a \in K\setminus K'. 
\]
\end{fact}
\begin{proof}
Let $a \in K \minus K'$. Since $g$ is an extreme point of $B$, Theorem~\ref{th1} (item~3) implies that $f$ is an extreme point of $B$, \mbox{\it i.e.} $\vert g \vert \equiv \vert f\vert \equiv 1 $, and it also implies that $F(f(a)f)=f(a)g$.\\
We have $\Vert \ind _{\{a\}}-f(a)f\Vert =1$. Since $F$ is non-expansive and $F(f(a)f)=f(a)g$, it follows that $\Vert F(\ind _{\{a\}})-f(a)g\Vert \leqslant 1$. This implies in particular that $f(a)g(\sigma_0(a))$ and $F(\ind_{\{a\}})(\sigma_0(a))$ have the same sign because $\vert f(a)g(\sigma_0(a))\vert =1$ and $F(\ind_{\{a\}})(\sigma_0(a)) \neq 0$, so we get $f(a)g(\sigma_0(a))= \mathrm{sgn}\bigl(F(\ind_{\{a\}})(\sigma_0(a))\bigr)$.\\
Applying that to $g=\ind$ and (hence) $f=\alpha$, we see that $\alpha (a)=\mathrm{sgn}\bigl(F(\ind_{\{a\}})(\sigma_0(a))\bigr)$.\\
Altogether, for an arbitrary extreme point $g$ and $f=F^{-1}(g)$, we obtain $f(a)g(\sigma_0(a))=\alpha (a)$, \mbox{\it i.e.} $g(\sigma_0(a))=\alpha (a)f(a)$. 
\end{proof}
\smallskip
\begin{step2} The map $\sigma_0:K\setminus K'\to K\setminus K'$ is bijective.
\end{step2}
\begin{proof} Let us first show that $\sigma_0$ is injective. 
Let $a,b \in K \minus K'$ be such that $\sigma_0(a)=\sigma_0(b)$. With the notation of Step~1, since $c_a$ and $c_b$ are continuous and since $c_a(-1)c_a(1)<0$ and $c_b(-1)c_b(1)<0$, it follows that $c_a([-1,1])=[u_a,v_a]$ and $c_b([-1,1])=[u_b,v_b]$, where $u_a<0<v_a$ and $u_b<0<v_b$. Hence, there exist $t_1, t_2 \in [-1,1]\minus \{0\}$ such that $c_a(t_1)=c_b(t_2)$. And since $\sigma_0(a)=\sigma_0(b)$, we get $F(t_1\ind_{\{a\}})=F(t_2\ind_{\{b\}})$, and this implies  that $t_1\ind_{\{a\}}=t_2\ind_{\{b\}}$ since $F$ is injective. Therefore, $t_1=t_2$ and $a=b$ since $t_1$ and $t_2$ are non-zero.

\smallskip
Now, let us show that $\sigma_0$ is surjective. Towards a contradiction, suppose that there exists $y \in K\setminus K'$ such that $y \neq \sigma_0(x)$ for all $x \in K \setminus K'$. \\
Let $h$ be the following extreme point of $B$: 
\[
h(x)= \begin{cases} 1 &\text{ if } x \neq y,\\
-1 &\text{ if } x=y.
\end{cases}
\]
By Fact~\ref{keyformula},  we have for every $a\in K\setminus K'$: 
\[ F^{-1}(h)(a)=\alpha(a)h(\sigma_0(a))=\alpha(a)\qquad \hbox{because $\sigma_0(a) \neq y$.}\]
 
 \noindent
Since $K \minus K'$ is dense in $K$, we get $F^{-1}(h)= \alpha =F^{-1}(\ind)$. Hence, $h= \ind$, which is a contradiction.
\end{proof}

\smallskip
\begin{step3} The map $\sigma_0:K\setminus K'\to K\minus K'$ extends to a continuous map $\sigma :K\to K$.
\end{step3}
\begin{proof} Since $K\setminus K'$ is dense in $K$, it is enough to prove the following: if $z\in K'$ and if $(a_d)_{d\in D}$ is any net in $K\setminus K'$ such that $a_d \longrightarrow z$, then the net $\bigl(\sigma_0(a_d)\bigr)$ has at most one cluster point. Indeed, since $K$ is compact we may then define $\sigma(z)=\lim \sigma_0(a_d)$ for any net in $K\setminus K'$ such that $a_d \longrightarrow z$ (the limit will not depend on $(a_d)$), and the map $\sigma :K\to K$ will be continuous. \\So let us fix $z \in K'$ and a net $(a_d)_{d\in D}\subset K\setminus K'$ such that  $a_d \longrightarrow z$.\\
Suppose that the net $\bigl(\sigma_0(a_d)\bigr)$ admits two different cluster points $u$ and $v$, \mbox{\it i.e.} there exist two subnets $(a_{d_i})_{i\in I}$ and $(a_{d_j})_{j\in J}$ of $(a_d)$ such that 
\[ \sigma_0(a_{d_i})\longrightarrow u\qquad{\rm and}\qquad \sigma_0(a_{d_j})\longrightarrow v.\]
Since $K$ is zero-dimensional, one can find a clopen neighbourhood $U$ of $u$ such that $v\notin U$. Let $g$ be the following extreme point of $B$:
\[
g(x)= \begin{cases} 1 &\text{ if } x \in U,\\
-1 &\text{ otherwise}.
\end{cases}
\]

\noindent 
By the definition of $g$, we have
\[ g(\sigma_0(a_{d_i}))=1\quad{\rm and}\quad g(\sigma_0(a_{d_j}))=-1 \qquad \hbox{for all large enough $i,j$}.\]

\noindent
Denoting $F^{-1}(g)$ by $f$, we obtain from Fact~\ref{keyformula} that $f(a_{d_i})=\alpha(a_{d_i})\;{\rm and}\; f(a_{d_j})=- \alpha(a_{d_j})$ for all large enough $i,j$. Since $f$ and $\alpha$ are continuous and $a_d \longrightarrow z$, it follows that $\alpha(z)=f(z)=-\alpha(z)$, which is a contradiction since $\vert \alpha(z)\vert=1$. Thus, we have shown that $\sigma_0$ extends to a continuous map $\sigma :K\to K$.
\end{proof}

\smallskip Now, let us state the following consequence of Fact~\ref{keyformula}.

\begin{fact}\label{keyformulabiz} If $g$ is any extreme point of $B$, then 
\[ F^{-1}(g)(x)= \alpha(x)g(\sigma(x))\qquad\hbox{for all}\; x \in K. 
\]
\end{fact}
\begin{proof}
Fact~\ref{keyformula} says that $F^{-1}(g)(x)=\alpha(x)g(\sigma(x))$, for all $x \in K \minus K'$, and this completes the proof since $\alpha$ and $\sigma$ are continuous and $K \minus K'$ is dense in $K$.
\end{proof}

\smallskip
\begin{step4} The map $\sigma:K\to K$ is a homeomorphism.
\end{step4}
\begin{proof}The map $\sigma:K\to K$ is surjective because $\sigma(K)$ is a compact set containing $\sigma_0(K\setminus K')=K\setminus K'$ and  $K\setminus K'$ is dense in $K$.

\smallskip
It remains to prove that $ \sigma$ is injective.\\
First, we observe that  $\sigma(K') \subset K'$. Indeed, let $b \in K'$. Since $K\setminus K'$ is dense in $K$, any neighbourhood of $b$ contains infinitely points of $K\setminus K'$. Since $\sigma$ is continuous and $\sigma_{| K\setminus K'}=\sigma_0$ is injective, it follows that any neighbourhood of $\sigma(b)$ contains infinitely points of $K$; and hence $\sigma(b)\in K'$.\\
So the map $\sigma :K\to K$ is surjective, it maps $K\setminus K'$ bijectively onto itself by Step~2, and $\sigma(K')\subset K'$. Since $K'$ is a finite set, it follows that $\sigma$ is injective.

\smallskip
Therefore, $\sigma : K \rightarrow K $ is a continuous bijection, hence a homeomorphism since $K$ is compact. 
\end{proof}

\smallskip
\begin{step5} {\rm 1), 2)} and {\rm 3)} hold true for every $f \in D:= \{g \in B ;\; \forall z \in K',\ g(z) \neq 0 \}$.
\end{step5}
\begin{proof}
Let $f \in D$ and let 
\[ S=\{x \in K ;\;  f(x)\neq 0 \}.\]
Since $f\in D$, we see that $S$ is an open set containing $K'$, and hence a clopen subset of $K$ (with finite complement). \\
Since $\sigma$ is a homeomorphism and $f \in D$, we may define the following extreme point of $B$: 
\[ \psi(\sigma(x))= \begin{cases}
  \alpha(x)\, \mathrm{sgn}\bigl(f(x)\bigr) &\text{if } x \in S , \\1 &\text{otherwise}. 
   \end{cases}\]
By Fact~\ref{keyformulabiz}, we have  
\[ F^{-1}(\psi)(x)= \begin{cases}
  \mathrm{sgn}\bigl(f(x)\bigr) &\text{if } x \in S , \\ \alpha(x) &\text{otherwise}. 
   \end{cases}\] 
Hence, $\Vert F^{-1}(\psi)-f\Vert \leqslant 1$. Since $F$ is non-expansive, it follows that $\Vert \psi -F(f)\Vert \leqslant 1$, and this implies 2) and 3). 

\smallskip
To prove 1), consider $y \in K\setminus K'$ such that $f(y)=0$ (then $y \notin S$). \\
We know that $\Vert \psi -F(f)\Vert \leqslant 1$, so $F(f)(\sigma(y))\geqslant 0$ since $\psi(\sigma(y))=1$. \\
Now, let us define the following extreme point of $B$: 
\[ \phi(x)=\begin{cases}
    \psi(x) &\text{if }x \neq \sigma(y), \\  -1 &\text{if }x =\sigma(y).
    \end{cases}
    \]
We have, as before, that  $\Vert \phi -F(f)\Vert \leqslant 1$, which implies that $F(f)(\sigma(y))\leqslant 0$. \\
Therefore, $F(f)(\sigma(y))=0$.
\end{proof}

\smallskip
We have just proved 1), 2) and 3) for all $f\in D$. Since $D$ is dense in $B$ by Remark~\ref{dense} and $F$ is continuous, it follows immediately that 2) and 3) hold true for every $f\in B$. Finally, if $f\in B$ and if $f(a)=0$ for some $a\in K\setminus K'$, then $f$ can be approximated by elements $g$ of $D$ such that $g(a)=0$. So 1) holds true as well for any $f\in B$.
\end{proof}

\smallskip
\begin{remark} In the proof of Lemma~\ref{l1}, the fact that $K'$ is a finite set was used just once, at the end of the proof of Step~4, to ensure that the map $\sigma$ is injective; and this assumption will not be needed anywhere else in the proof of Theorem~\ref{main}.

However, note that if $K$ has infinitely many accumulation points, then it may be possible to construct a surjective continuous map $\sigma :K\to K$ such that $\sigma$ maps bijectively $K\setminus K'$ onto itself, $\sigma(K')=K'$ and yet $\sigma$ is not injective. Consider for example the simplest compact space $K$ with infinitely many accumulation points, \mbox{\it i.e.} $K=\{ x_{n,m};\; n,m\in\N\}\cup\{ x_{n,\infty};\; n\in\N\} \cup \{x_{\infty,\infty}\}$ where all points are distinct, $x_{n,m}\longrightarrow x_{n,\infty}$ as $m\longrightarrow\infty$ for each $n\in\N$, and $x_{n,\infty}\longrightarrow x_{\infty,\infty}$ as $n\longrightarrow\infty$. Define $\sigma:K\rightarrow K$ as follows: 
\begin{enumerate}
\item[-] $\sigma(x_{\infty,\infty})=x_{\infty,\infty}$,
\item[-] $\sigma(x_{n,\infty})=x_{(n-1),\infty}$ for all $n\geqslant 2$, and $\sigma(x_{1,\infty})=x_{1,\infty}$,
\item[-] $\sigma(x_{n,m})=x_{(n-1),m}$ for all $n\geqslant 3$ and $m\in\N$,
\item[-] $\sigma(x_{2,m})=x_{1, (2m-1)}$ and $\sigma(x_{1,m})=x_{1,2m}$ for all $m\in\N$.
\end{enumerate}

\noindent Then, $\sigma$ has the properties stated above.
\end{remark}

\medskip
With the notation of Lemma~\ref{l1}, let us consider the linear isometric bijection $\mathcal{J}: C(K) \rightarrow C(K)$ defined as follows: for every $f\in C(K)$ and $x\in K$,
\[ \mathcal{J}f( \sigma(x))= \alpha (x)f(x).
\]

\noindent In other words,
\[ \mathcal J= M_{\alpha\circ\sigma^{-1}} C_{\sigma^{-1}},\]
where $M_{\alpha\circ\sigma^{-1}}$ is the isometric multiplication operator defined by  the unimodular function $\alpha\circ\sigma^{-1}$ and $C_{\sigma^{-1}}$ is the isometric composition operator defined by the homeomorphism $\sigma^{-1}$.

\medskip
Considering $\mathcal J_{| B}$ as an isometric bijection of $B$ onto itself, the proof of Theorem~\ref{main} will be complete if we can show that $\mathcal J_{| B}=F$. In other words, setting 
\[ \Phi = (\mathcal J_{| B})^{-1}\circ F,\]
our aim is now to show that $\Phi =Id_B$. Note that by definition of $\Phi $, we have $\Phi=M_\alpha C_\sigma\circ F$, \mbox{\it i.e.}
\begin{equation}\label{Ftilde} \Phi (f)(x)=\alpha(x) F(f)(\sigma(x))\end{equation}
for every $f\in B$ and $x\in K$. The next lemma summarizes all what we need to know about $\Phi $ in what follows.

\begin{lemma}\label{l2}
The map $\Phi $ is a non-expansive bijection of $B$ onto itself. Moreover, for any $f, g \in B$ and every $x \in K\setminus K'$, the following facts hold true.
\begin{enumerate}[label={\rm \arabic*)}]
\item\label{1} If $f(x)=0$, then $\Phi (f)(x)=0$.
\item\label{2} If $f(x)<0$, then $\Phi (f)(x) \leqslant 0$.
\item\label{3} If $f(x)>0$, then $\Phi (f)(x) \geqslant 0$. 
\item\label{4} If $g(x)<0$, then $\Phi ^{-1}(g)(x) <0$.
\item\label{5}  If $g(x)>0$, then $\Phi ^{-1}(g)(x) >0$.
\item\label{6} If $g$ is an extreme point of $B$, then $\Phi(g)=g$. 
\end{enumerate}
\end{lemma}
\begin{proof} Properties 1), 2) and 3) are immediate consequences of (\ref{Ftilde}) and Lemma~\ref{l1}. If $g(x)<0$ then, applying \ref{1} and \ref{3} to $f:=\Phi ^{-1}(g)$, we see that we cannot have $\Phi ^{-1}(g)(x)\geqslant 0$; which proves \ref{4}. The proof of \ref{5} is the same.\\
Let us now prove \ref{6}. Assume that $g$ is an extreme point of $B$. With the notation of Lemma~\ref{l1}, for every $x \in K$, we have 
\begin{equation}\label{Jg} \mathcal{J}g(\sigma(x))=\alpha(x)g(x). \end{equation}
Since $g$ is an extreme point of $B$, (\ref{Jg}) implies that $\mathcal{J}g$ is also an extreme point of $B$. Hence, by Fact~\ref{keyformulabiz}, we get $F^{-1}(\mathcal{J}g)(x)=g(x)$, for every $x \in K$. This implies that $\Phi^{-1}(g)=g$, which is equivalent to $\Phi(g)=g$.
\end{proof}

\medskip To summarize this subsection: we have defined a non-expansive bijection $\Phi:B\to B$ satisfying the properties stated in Lemma~\ref{l2}, and Theorem~\ref{main} will be proved if we can show that $\Phi=Id_B$.

\smallskip From now on, we will assume that \emph{$K$ is metrizable}. This assumption will be needed to ensure that the closed set $K'$ is a ``zero set'', \mbox{\it i.e.} there is a continuous function $\chi:K\to \R$ such that $K'=\{ \chi =0\}$. And this will be used only in the proof of Fact~\ref{l12} below.

\smallskip On the other hand, the fact that $K'$ is a finite set will not be needed in the remaining of the proof.

\subsection{The main lemma} The fact that $\Phi =Id_B$ will follow  easily from Lemma~\ref{l17} below. Let us first introduce some notation. Given $f \in B$, we define 
\begin{center}
$u_{0,f}= \min \,\{\vert f(z)\vert ;\;  z \in K'\}$
\end{center}
and for any $n\in\Z_{+}$, we set
\[ u_{n,f}= \dfrac{n+u_{0,f}}{n+1}\cdot\]

\medskip

\begin{lemma}\label{l17}
For each $n \in \Z_+$, we have the following: for each $f \in B$ such that $u_{0,f} \notin \{0,1\}$ and $x \in K\setminus K'$,
\begin{enumerate}[label={\rm \arabic*)}]
\item if $\vert f(x)\vert < u_{n,f}$, then $\Phi (f)(x)=f(x)$, 
\item if $f(x) \geqslant u_{n,f}$, then $\Phi (f)(x) \in [u_{n,f},f(x)]$, 
\item if $f(x) \leqslant -u_{n,f}$, then $\Phi (f)(x) \in [f(x),-u_{n,f}]$. 
\end{enumerate}
\end{lemma}

\medskip Assume that we have been able to prove Lemma~\ref{l17}. Then we easily deduce
\begin{corollary}\label{l18}
For each $f \in B$ such that $u_{0,f}\notin \{0,1\}$, we have $\Phi (f)=f$. 
\end{corollary}
\begin{proof}
Let $f \in B$ be such that $u_{0,f }\notin \{0,1\}$ and let $x \in K \setminus K'$. 
If $\vert f(x)\vert <1$ then, since $u_{n,f}\longrightarrow 1$ as $n\longrightarrow\infty$, there exists $n \in \Z_+$ such that $\vert f(x)\vert <u_{n,f}$, so $\Phi (f)(x)=f(x)$ by Lemma~\ref{l17} (item~1). If $f(x)=1$, then $f(x)\geqslant u_{n,f}$ for all $n\in\Z_+$, and hence $\Phi(f)(x)=1=f(x)$ by Lemma~\ref{l17} (item~2). Similarly, if $f(x)=-1$ then $\Phi(f)(x)=-1=f(x)$. Hence, for every $x \in K \minus K'$, we have $\Phi(f)(x)=f(x)$, and this implies that $\Phi(f)=f$ since $K \minus K'$ is dense in $K$. \end{proof}

\medskip Using Corollary~\ref{l18}, it is very easy to show that $\Phi=Id_B$, and hence to conclude the proof of Theorem~\ref{main}. Let 
\[D= \bigl\{f \in B ;\;  \forall z \in K',\ f(z) \notin\{ 0,1,-1\}\bigr\}.\]

\noindent
The set $D$ is contained in $\bigl\{ f\in B;\; u_{0,f}\notin\{ 0,1\}\bigr\}$; so, by Corollary~\ref{l18}, we have $\Phi(f)=f$ for all $f\in D$. Since $\Phi$ is continuous and $D$ is dense in $B$ by Remark~\ref{dense}, it follows that $\Phi=Id_B$.

\medskip 
In view of the previous discussion, our only task is now to prove Lemma~\ref{l17}; which will occupy us for a while.
\subsection{Proof of the main lemma: $n=0$} The case $n=0$ of Lemma~\ref{l17} will be deduced from the next two lemmas.
\begin{lemma}\label{l5}
For any $f, g \in B$ and every $x \in K\setminus K'$, the following implications hold true.
\begin{enumerate}[label={\rm \arabic*)}]
\item If $f(x)<0$, then $\Phi (f)(x) \in [f(x),0]$,
\item if $f(x)>0$, then $\Phi (f)(x) \in [0,f(x)]$. 
\item If $g(x)<0$, then $\Phi ^{-1}(g)(x)\in [-1, g(x)] $,
\item if $g(x)>0$, then $\Phi ^{-1}(g)(x)\in [g(x),1]$, 
\item if $g(x)=\pm 1$, then $\Phi ^{-1}(g)(x)=g(x)$.
\end{enumerate}
\end{lemma}
\begin{proof} 
To prove 1), let $f \in B$ and let $x \in K\setminus K'$ be such that $f(x)<0$: we want to show that $\Phi(f)(x)\in [f(x),0]$. Note that by Lemma~\ref{l2}, we already know that $\Phi (f)(x) \leqslant 0$; so we only need to show that $\Phi (f)(x) \geqslant f(x)$. \\
Let us first assume that $f(z)\neq 0$ for every $z\in K$. Then, we can define $g\in B$ as follows:
\[
g(y)= \begin{cases} 1&\text{if }y=x,\\
{\rm sgn}\bigl(f(y)\bigr) &\hbox{if $y\neq x$} .
\end{cases}
\]

\noindent
We have $\Vert f-g\Vert =1- f(x)$. Since $g$ is an extreme point of $B$, Lemma~\ref{l2} implies that $\Phi(g)=g$. Hence, $\Vert \Phi(f)-g\Vert = \Vert \Phi(f)-\Phi(g)\Vert \leqslant 1- f(x)$, because $\Phi$ is non-expansive. \\
Then, \[\begin{aligned}
    \vert \Phi (f)(x)-g(x)\vert \leqslant 1-f(x) &\implies 1- \Phi (f)(x) \leqslant 1-f(x) \\&\implies \Phi (f)(x) \geqslant f(x).
\end{aligned}\]

\noindent Now consider any $f\in B$ such that $f(x)<0$. One can find a sequence $(f_n)_{n\in\N}\subset B$ such that $f_n\longrightarrow f$, $f_n(x)=f(x)$ for all $n$ and the functions $f_n$ never vanish on $K$. Then, $\Phi(f_n)(x)\in [f_n(x),0]$ for all $n$, and hence $\Phi(f)(x)\in [f(x),0]$ because $\Phi$ is continuous.

\smallskip\noindent
 The proof of 2) is the same. (First consider $f\in B$ such that $f(x)>0$ and $f(z)\neq 0$ for every $z\in K$. Define $g$ as above but with $g(x)=-1$. Then conclude by approximation.)

\medskip
Let us prove 3) and 4). If  $g \in B$ and $x \in K\setminus K'$, then
\begin{eqnarray*}
   g(x)<0 &\implies&  \Phi ^{-1}(g)(x)<0 \qquad\hbox{by Lemma~\ref{l2}}\\&\implies&\Phi (\Phi ^{-1}(g))(x) \in [\Phi ^{-1}(g)(x),0]\qquad\hbox{by (1)}\\&\implies& -1\leqslant \Phi ^{-1}(g)(x) \leqslant g(x).
\end{eqnarray*}
The case $g(x) >0$ is similar.

\medskip Finally, 5) follows from 3) and 4).
\end{proof}

\begin{lemma}\label{l7}
If $f \in B$ is such that the set $S_f = \bigl\{x \in K ;\;  f(x) \notin \{-1,1\}\bigr\}$ is finite, then $\Phi (f)=f$.
\end{lemma} 
\begin{proof} We first observe that for any $f\in B$ such that $S_f$ is finite, we have
\[ K' \cap S_f = \emptyset.\]
Indeed, suppose that there exists $x_0\in K' \cap S_f$; so any neighbourhood of $x_0$ contains infinitely many points, and  $\vert f(x_0)\vert<1$. Since $f$ is continuous, it follows that there exist infinitely many $x\in K$ such that $f(x)\notin\{ -1,1\}$, which is a contradiction. 

\medskip
To prove the lemma, we proceed by induction on the number of elements of $S_f$. 

\smallskip\noindent
If $S_f$ is empty, then $f$ is an extreme point of $B$, and Lemma~\ref{l2} implies that $\Phi(f)=f$.

\smallskip\noindent
Now, let $M$ be a non-negative integer and suppose that the lemma has been proved for all $f\in B$ such that  $\# S_f\leqslant M$. Let us show that the claim also holds for all $f\in B$ such that $\#S_f=M+1$. \\
Let $f \in B$ be such that $\# S_f= M+1$. We need to show that $\Phi (f)=f$, \textit{ i.e.} $\Phi ^{-1}(f)=f$.\\
Lemma~\ref{l5} implies that 
\begin{equation}\label{formulaS}
    \Phi  ^{-1}(f)(x)=f(x)\qquad\hbox{for all } x\in K\setminus(K'\cup S_f).
\end{equation}
Moreover, if $x \in K'$, then there exists a net $(x_d)_{d\in D}$ in $K\setminus (K'\cup S_f)$ such that $x_d \longrightarrow x$ (because $K\setminus K'$ is dense in $K$ and $S_f$ is finite). 
The continuity of $\Phi ^{-1}(f)$ implies that $\Phi  ^{-1}(f)(x_d) \longrightarrow \Phi  ^{-1}(f)(x)$, \mbox{\it i.e.}  $f(x_d) \longrightarrow \Phi  ^{-1}(f)(x)$ by (\ref{formulaS}). Since $f$ is continuous, it follows that $\Phi^{-1}(f)(x)=f(x)$.\\
Hence, we have shown that $\Phi  ^{-1}(f)(x)=f(x)$ for all $x \in K \minus S_f$. It remains to show that $\Phi  ^{-1}(f)(x)=f(x)$ for all $x \in S_f$. \\
Let $x \in S_f$ (then $x \in K\setminus K'$ since $S_f\cap K'=\emptyset$). We define the following two functions $g$ and $h$ from $K$ to $\R$:
\[
    g(x)=1\;,\quad  h(x)=-1 \quad\hbox{and} \quad g(y)=h(y)= \Phi ^{-1}(f)(y)\; \;\hbox{for all}\;y \neq x. 
\]
Since $x\in K\setminus K'$, these functions $g$ and $h$ are continuous, and hence $g,h\in B$. \\
Note that $S_g= \bigl\{x \in K ;\;  g(x) \notin \{-1,1\}\bigr\} $ and $S_h= \bigl\{x \in K ;\;  h(x) \notin \{-1,1\}\bigr\} $ are contained in the set $S_f \minus \{x\}$. 
Since $\#(S_f \minus \{x\})=M$ (because $\# S_f=M+1$ and $x \in S_f$), it follows that  $\#S_h \leqslant M$ and $\# S_g \leqslant M$. 
Therefore, we can apply the induction hypothesis to $g$ and $h$ to obtain $\Phi (g) =g$ and $\Phi (h) =h$. \\
Since $g$, $h$ and $\Phi ^{-1}(f)$ coincide on $K \minus \{x\}$, we have 
\[ \Vert \Phi  ^{-1}(f)-g\Vert =\vert \Phi ^{-1}(f)(x)-g(x)\vert, \]
and 
\[ \Vert \Phi  ^{-1}(f)-h\Vert =\vert \Phi ^{-1}(f)(x)-h(x)\vert .\]
Since $\Phi $ is non-expansive, $\Phi (g) =g$ and $\Phi (h) =h$, it follows that
\begin{equation}\label{ineq1}
\vert f(x)-g(x)\vert \leqslant\Vert f-g\Vert \leqslant \Vert \Phi ^{-1}(f)-g\Vert = \vert \Phi ^{-1}(f)(x)-g(x)\vert ,
\end{equation}
and 
\begin{equation}\label{ineq2}
\vert f(x)-h(x)\vert \leqslant\Vert f-h\Vert \leqslant \Vert \Phi ^{-1}(f)-h\Vert = \vert \Phi ^{-1}(f)(x) -h(x)\vert .
\end{equation}
Hence, 
\[(\ref{ineq1}) \implies 1-f(x) \leqslant 1- \Phi ^{-1}(f)(x) \implies f(x) \geqslant \Phi ^{-1}(f)(x) ,\]
and
\[ (\ref{ineq2})\implies f(x)+1 \leqslant \Phi ^{-1}(f)(x) +1 \implies f(x) \leqslant \Phi ^{-1}(f)(x).\]

\smallskip\noindent
Therefore, $\Phi  ^{-1}(f)(x)=f(x)$.
\end{proof}

\medskip We are now in position to prove the case $n=0$ of Lemma~\ref{l17}. We prove in fact a slightly more general result since we do not assume that $u_{0,f}\neq 1$.
\begin{lemma}\label{l8}
Let $f \in B$ be such that $h =u_{0,f}= \min\, \{\vert f(z)\vert ;\;  z \in K'\}\neq 0$. For any $x \in K\setminus K'$, we have the following implications: \begin{enumerate}[label={\rm \arabic*)}]
\item if $\vert f(x)\vert <h$, then $\Phi (f)(x)=f(x)$, 
\item if $f(x) \geqslant h$, then $\Phi (f)(x) \in [h,f(x)]$, 
\item if $f(x) \leqslant -h$, then $\Phi (f)(x) \in [f(x),-h]$. \end{enumerate}
\end{lemma}
\begin{proof}
1) Let $x \in K\setminus K'$ be such that $\vert f(x)\vert  <h$. \\
If $f(x)=0$, then Lemma~\ref{l2} implies that $\Phi (f)(x) =0$; so we assume that 
 $f(x) \neq 0$. In what follows, we set $\varepsilon = h -\vert f(x)\vert  >0$. \\
Since $f: K \rightarrow \R$ is continuous, $u_{0,f}\neq 0$ and $K$ is compact and zero-dimensional, one can find clopen sets $V_1,\dots ,V_N\subset K$ such that 
\begin{itemize}
\item[-] $V_i\cap K'\neq\emptyset$ for all $i$ and $\bigcup_{i=1}^N V_i\supset K'$,
\item[-] the $V_i$ are pairwise disjoint and do not contain $x$,
\item[-] $f$ has constant sign $\varepsilon_i$ on each $V_i$,
\item[-] ${\rm diam}\bigl(f(V_i)\bigr)<\varepsilon$ for $i=1,\dots ,N$.
\end{itemize}

\smallskip\noindent
Now, we define $g \in B$ as follows: 
\[
    g(y)= \begin{cases}
    \text{sgn}\bigl(f(x)\bigr) &\text{if }y=x, \\
    \varepsilon_i &\text{if } y \in V_i\;\hbox{for some}\; i \in \{1,\dots,N\},\\
    f(y) &\text{if } y \notin \biggl(\displaystyle \bigcup _{i=1}^{N}V_i \cup \{x\}\biggr).
     \end{cases}
\]
Since $K \minus  \bigcup _{i=1}^{N}V_i$ is a closed subset of $K$ that does not contain accumulation points, and since $K$ is compact, we see that $K \minus  \bigcup _{i=1}^{N}V_i$ is finite. Hence, we can apply Lemma~\ref{l7} to $g$, and we obtain \[ \Phi  (g)=g.\]
Let us show that $\Vert f-g\Vert  \leqslant 1-\vert f(x)\vert$, \mbox{\it i.e.} $\vert f(y)-g(y)\vert\leqslant 1-\vert f(x)\vert$ for all $y\in K$. \\
- If $y \notin \left( \bigcup _{i=1}^{N}V_i \cup \{x\}\right)$, then $\vert f(y)-g(y)\vert =0$. \\
- If $y=x$, then $\vert f(y)-g(y)\vert =1-\vert f(x)\vert $. \\
- If $y \in V_i$ for some $i$, choose $z\in K'\cap V_i$. Then, $\vert f(y)-f(z)\vert <\varepsilon$ and  ${\rm sgn} \bigl(f(y)\bigr)=\varepsilon_i$. So, by the definition of $g$ and $h$, we get
\begin{align*}
\vert f(y)-g(y)\vert &= 1-\vert f(y)\vert\\
&<1-\vert f(z)\vert +\varepsilon\\
&\leqslant 1-h+\varepsilon=1-\vert f(x)\vert.
\end{align*}

\noindent
So we have indeed $\Vert f-g\Vert =1-\vert f(x)\vert $. \\
Since $\Phi $ is non-expansive and $\Phi (g)=g$, it follows that $\Vert \Phi(f)-g\Vert \leqslant 1-\vert f(x)\vert$, and in particular
\begin{align}\label{ineqphi}  \vert \Phi (f)(x)-g(x)\vert \leqslant 1-\vert f(x)\vert .  \end{align}
If $f(x)<0$, then  $g(x)=-1$ and hence
\[ (\ref{ineqphi}) \implies 1+\Phi (f)(x) \leqslant 1+ f(x) \implies \Phi (f)(x) \leqslant f(x).\]
Moreover, Lemma~\ref{l5} implies that $\Phi (f)(x) \geqslant f(x)$. So we get $\Phi (f)(x)=f(x)$. \\
Similarly, if $f(x)>0$ then $g(x)=1$; so, using (\ref{ineqphi}), we obtain $\Phi(f)(x) \geqslant f(x)$, and hence $\Phi(f)(x)=f(x)$ by Lemma~\ref{l5}.

\medskip
2) Let $x\in K\setminus K'$ be such that $f(x) \geqslant h$.\\
Since $f(x) \geqslant h >0$, Lemma~\ref{l5} implies that $\Phi (f)(x) \leqslant f(x)$. So 
it remains to show that $\Phi (f)(x) \geqslant h$.\\
Let $\varepsilon >0$ be arbitrary. \\
Let us choose $V_1,\dots, V_N$ as in the proof of 1), and define $g \in B$ as in the proof of 1), so that $\Phi (g)=g$. \\
Let us show that $\Vert f-g\Vert \leqslant 1-h+\varepsilon$. \\
- If $y \notin \left( \bigcup _{i=1}^{N}V_i \cup \{x\}\right)$, then $\vert f(y)-g(y)\vert =0$. \\
- If $y=x$, then $\vert f(y)-g(y)\vert =1-f(x) \leqslant 1-h$. \\
- If $y \in V_i$ for some $i$, take $z\in V_i\cap K'$. Then, $\vert f(y)-f(z)\vert <\varepsilon$  and  ${\rm sgn} \bigl(f(y)\bigr)=\varepsilon_i$, so we get as above $\vert f(y)-g(y)\vert <1-h + \varepsilon$. This shows that $\Vert f-g\Vert \leqslant 1-h+\varepsilon$.\\
Since $g(x)=1$ and $\Phi $ is non-expansive, we get
\[\begin{aligned}
\Vert \Phi (f)-g\Vert =\Vert \Phi (f)-\Phi (g)\Vert  \leqslant1-h+\varepsilon &\implies \vert \Phi (f)(x)-g(x)\vert \leqslant1-h + \varepsilon \\&\implies 1-\Phi (f)(x) \leqslant1-h + \varepsilon \\&\implies \Phi (f)(x) \geqslant h -\varepsilon.
\end{aligned}\]
Since $\varepsilon>0$ is arbitrary, it follows that $\Phi (f)(x) \geqslant h$.

\medskip
3) The proof is similar to that of 2). 
\end{proof} 

\medskip
For future use, we state two consequences of Lemma~\ref{l8}.  

\begin{corollary}\label{l9}
For each $f\in B$, we have  $u_{0,f}=u_{0,\Phi (f)}$.
\end{corollary}
\begin{proof}
Let us first show that if $u_{0,f}=0$ then $u_{0,\Phi(f)}=0$. \\
Let $z_0 \in K'$ be such that $f(z_0)=0$. Lemma~\ref{l2} (item~1) and Lemma~\ref{l5} (items~1 and~2) imply that $\Phi(f)(z_0) =0$. Then, $u_{0,\Phi(f)}=0$. \\
Now, let us show that if $u_{0,f}\neq 0$ then $u_{0,\Phi(f)}=u_{0,f}$. \\
On the one hand, let $z_{0} \in K'$ be such that $u_{0,f}=\vert f(z_{0})\vert $. 
Lemma~\ref{l8} (items~1, 2 and 3) implies that $\Phi (f)(z_{0})=f(z_{0})$. 
Hence, $u_{0,\Phi (f)}\leqslant u_{0,f}$. \\
On the other hand, for any $z\in K'$, we have $\vert f(z)\vert  \geqslant u_{0,f}$. 
If $\vert f(z)\vert  = u_{0,f}$, then Lemma~\ref{l8} (items~1, 2 and 3) implies that $\Phi (f)(z)=f(z)$, whereas if
 $\vert f(z)\vert  > u_{0,f}$, then Lemma~\ref{l8} (items~2 and 3) implies that $\vert \Phi (f)(z)\vert  \geqslant  u_{0,f}$. 
Hence, we have $\vert \Phi (f)(z)\vert  \geqslant  u_{0,f}$ for all $z\in K'$,  so that $ u_{0,f} \leqslant u_{0,\Phi (f)} $. \\
Altogether, we obtain $u_{0,f}= u_{0, \Phi (f)}$.
\end{proof}

\smallskip
\begin{corollary}\label{l11}
Let $g \in B$ be such that $u_{0,g}\neq 0$ and $x \in K\setminus K'$.  
If $\vert g(x)\vert <u_{0,g}$, then $\Phi  ^{-1}(g)(x)=g(x)$.
\end{corollary}
\begin{proof}
Denote $\Phi ^{-1}(g)$ by $f$. Then, $u_{0,f}=u_{0,g} \neq 0$ by Corollary~\ref{l9}. \\
Three cases may occur: $\vert f(x)\vert <u_{0,f}$, $f(x) \geqslant u_{0,f}$ or $f(x) \leqslant - u_{0,f}$. \\
If $f(x) \geqslant u_{0,f}$, then Lemma~\ref{l8} implies that $g(x) \geqslant u_{0,f}=u_{0,g}$, which is a contradiction. Similarly, if
 $f(x) \leqslant - u_{0,f}$, then Lemma~\ref{l8} gives a contradiction. 
So, we have $\vert f(x)\vert <u_{0,f}$. Hence, by Lemma~\ref{l8} again, $f(x)= \Phi (f)(x)$, \textit{i.e.} $\Phi  ^{-1}(g)(x)=g(x)$.
\end{proof}

\begin{remark} If we knew that $\Phi^{-1}$ is continuous then, in Corollary~\ref{l11}, it would be enough to assume that $\vert g(x)\vert \leqslant u_{0,g}$. Indeed, for any $n\in\N$, the function $g_n\in B$ defined by $g_n(x)=(1-\frac1n) g(x)$ and $g_n(y)=g(y)$ for $y\neq x$ satisfies $\vert g_n(x)\vert <u_{0,g}=u_{0,g_n}$, and $g_n\longrightarrow g$ as $n\longrightarrow\infty$. So, applying Corollary~\ref{l11} to $g_n$, we get $\Phi^{-1}(g)(x)=g(x)$.
\end{remark}

\subsection{A digression} In this subsection, we obtain some information on $\Phi(f)$ when $f\in B$ takes only a finite number of values. This will be needed for the proof of Lemma~\ref{l17}.

\medskip  
Let us fix some notation. 

\smallskip  We denote by $\mathcal V$ the family of all finite sequences $\mathbf V=(V_1,\dots ,V_N)$ of clopen subsets of $K$ such that the $V_i$ are pairwise disjoint, $V_i\cap K'\neq\emptyset$ for all $i$ and $\bigcup_{i=1}^N V_i\supset K'$. 

\smallskip\
If $\mathbf V=(V_1,\dots ,V_N)\in\mathcal V$ and $\mathbf h=(h_1,\dots ,h_N)\in [-1,1]^N$, we define
\[ B_{\mathbf V, \mathbf h}= \bigl\{f \in B ;\;  f_{|_{V_i}}\equiv h_i, \; \forall i \in \{1,\dots,N\}\bigr\}.\]

\smallskip We are going to prove the following lemma.
\begin{lemma}\label{l13}
Let $\mathbf V=(V_1,\dots ,V_N)\in\mathcal V$, and let $h\in (0,1)$. For any $\mathbf h=(h_1,\dots,h_N) \in \{-h,h\}^N$, we have 
 $\Phi _{|_{B_{\mathbf V,\mathbf h}}}=Id$.
\end{lemma}
\begin{proof} This will follow from the next two facts.

\begin{fact}\label{l10}
Let $\mathbf V=(V_1,\dots ,V_N)\in\mathcal V$, $\mathbf h =(h_1,\dots,h_N) \in [-1,1]^N$, and assume that $ h= \min \, \bigl\{\vert h_i\vert ;\;  i \in \{1,...,N\}\bigr\} \neq 0$. 
If $\Phi (B_{\mathbf V,\mathbf h})=B_{\mathbf V, \mathbf h}$, then $\Phi _{|_{B_{\mathbf V,\mathbf h}}}= Id$.
\end{fact}
\begin{proof}[Proof of Fact \ref{l10}] For any $f\in B_{\mathbf V,\mathbf h}$, let us define a function $J(f):K\to\R$ as follows: 
\[ J(f)(x)=f(x)\;\qquad\hbox{for all}\;x \notin \displaystyle\bigcup_{i=1}^{N}V_i,\]
and 
\[J(f)_{|_{V_i}} \equiv 0\;\qquad\hbox{for}\; i=1,\dots, N.\]

\noindent
Then, $J(f)\in B_{\mathbf V, \mathbf 0}$ where $\mathbf 0=(0,\dots ,0)$. Moreover, it is clear that 
$J:  B_{\mathbf V, \mathbf h}  \to  B_{\mathbf V,\mathbf 0} $ is an isometric bijection.

\smallskip
Since $\Phi  : B_{\mathbf V,\mathbf h} \rightarrow B_{\mathbf V,\mathbf h}$ is a bijection, we can define a bijective  map $\Psi _0 :B_{\mathbf V, \mathbf 0} \rightarrow B_{\mathbf V, \mathbf 0} $ as follows: 
\[ \Psi _0\bigl(J(f)\bigr)=J\bigl(\Phi (f)\bigr)\qquad\hbox{for all}\; f \in B_{\mathbf V, \mathbf h}.\]

\noindent
 Moreover, since $J$ is an isometry and $\Phi $ is non-expansive, $\Psi _0$ is non-expansive.
Now, $B_{\mathbf V,\mathbf 0}$ is the unit ball of the finite-dimensional space \[ X_{\mathbf V,\mathbf 0}=\bigl\{f \in C(K) ;\;  f_{|_{V_i}}\equiv 0, \; \forall i \in \{1,\dots,N\}\bigr\} ,\] so $B_{\mathbf V, \mathbf 0}$ is plastic. Therefore, $\Psi _0$ is in fact an isometry. Hence, by Theorem~\ref{th2},  $\Psi _0$ extends to a linear isometric bijection $\widetilde\Psi_0:X_{\mathbf V,\mathbf 0}\to X_{\mathbf V, \mathbf 0}$. \\
Let  \[ M=\Bigl\{x\in K\setminus K' ;\;  x \notin \bigcup_{i=1}^{N}V_i\Bigr\}.\] Since $K'\subset \bigcup_{i=1}^N V_i$, the set $M$ is finite. We write $M= \{x_1,\dots,x_n\}$.\\
Now, let $f \in B_{\mathbf V,\mathbf h}$, and let us show that $\Phi(f)=f$. \\
Since $h\neq 0$, for each $j \in \{1,\dots,n\}$ we may write $f(x_j)=\alpha_jh$ where $\alpha_j \in \R$.  Then,
\[ J(f)= \displaystyle\sum_{j=1}^{n}\alpha_j h\ind_{\{x_j\}}.\]
For each $j \in \{1,\dots,n\}$, we define $f_j \in B_{\mathbf V, \mathbf h}$ as follows:
\[\begin{aligned}
\begin{cases}
f_{j_{|_{V_i}}}\equiv h_i & \text{for }i = 1,\dots,N,\\
f_j(y)=0 &\text{for all } y \notin \biggl(\displaystyle\bigcup_{i=1}^{N}V_i\cup \{x_j\}\biggr), \\
f_j(x_j)=h.
\end{cases}
\end{aligned}\]

\noindent Note that $J(f_j)=h \ind_{\{x_j\}}$, for each $j \in \{1,\dots,n\}$. Moreover, 
by Lemma~\ref{l8} and since we are assuming that $\Phi (B_{\mathbf V, \mathbf h})=B_{\mathbf V,\mathbf h}$, we see that $\Phi (f_j)=f_j$ for all $j \in \{1,\dots,n\}$: indeed, we have $\Phi(f_j)\equiv h_i \equiv f_j$ on each set $V_i$, and $\Phi(f_j)(x)=f_j(x)$ for all $x\in K\setminus \bigcup_{i=1}^N V_i$ by Lemma~\ref{l8}.  
Hence, using the linearity of $\widetilde{\Psi}_0$, we get
\[\begin{aligned}
\Psi _0(J(f)) &= \widetilde\Psi  _0\biggl(\displaystyle\sum_{j=1}^{n}\alpha_j h\ind_{\{x_j\}}\biggr)\\
&=\displaystyle\sum_{j=1}^{n}\alpha_j\Psi _0\bigl(h \ind_{\{x_j\}}\bigr)\\&= \displaystyle\sum_{j=1}^{n}\alpha_j J\bigl(\Phi (f_j)\bigr)\\&= \displaystyle\sum_{j=1}^{n}\alpha_j J(f_j) \\&=\displaystyle\sum_{j=1}^{n}\alpha_j h \ind_{\{x_j\}}=J(f).
\end{aligned}\]
It follows that $J(\Phi (f))=J(f)$ and this implies that $\Phi (f)=f$. So we have shown that  $\Phi  _{|_{B_{\mathbf V, \mathbf h}}}=Id$.
\end{proof}
\begin{fact}\label{l12}
Let $\mathbf V=(V_1,\dots ,V_N)\in\mathcal V$ and $h\in (0,1)$. For any $\mathbf h =(h_1, \dots, h_N) \in \{-h,h\}^N$, we have 
 $\Phi ^{-1}(B_{\mathbf V, \mathbf h}) \subset B_{\mathbf V, \mathbf h}$.
\end{fact}
\begin{proof}[Proof of Fact \ref{l12}]
Let $g \in B_{\mathbf V, \mathbf h}$. Our goal is to show that $\Phi  ^{-1}(g) \in B_{\mathbf V, \mathbf h}$, \textit{i.e.} $\Phi  ^{-1}(g)_{|_{V_i}}\equiv h_i$, for every $ i \in \{1,\dots,N\}$. \\
Let $i_0 \in \{1,\dots,N\}$ and let  $x \in K\setminus K'$ be such that $x \in V_{i_0}$. 
We want to show that $\Phi ^{-1}(g)(x)= h_{i_0}$. 

\smallskip\noindent
Suppose for example that $h_{i_0} >0$, \textit{i.e.} $h_{i_0}=h$. 

\smallskip\noindent
Since $x \in V_{i_0}$, we have $g(x)=h$. 
Moreover, since $h>0$, Lemma~\ref{l5} implies that $\Phi  ^{-1}(g)(x) \geqslant h$. 
So, it remains to show that $\Phi ^{-1} (g)(x) \leqslant h$. 
For the sake of contradiction, suppose that $\Phi ^{-1} (g)(x) > h$. \\
Since $K$ is metrizable, the closed set $K'$ is also $G_\delta$, and hence it is a ``zero set''; so one can choose a continuous function $\chi : K\to\R$ such that 
\[ 0\leqslant \chi(z) <\frac12\quad\hbox{for all $z\in K$}\qquad{\rm and}\qquad K'=\{ z\in K;\; \chi(z)=0\}.\] 

\noindent Finally, 
let us denote $1-h$ by $v$.

\smallskip\noindent
We define $l\in B$ as follows: \[
\begin{aligned}
   l(y)= \begin{cases}
   1 &\text{if }y=x,\\
   \mathrm{sgn}(h_i)\Bigl(h+\frac{v}{2}\bigl(1-\chi(y)\bigr)\Bigr)  &\text{if }y\in V_i\quad\hbox{for some $i$ and}\; y\neq x,\\
     g(y)  &\text{if }\vert g(y)\vert <h, \\
    \mathrm{sgn}\bigl(g(y)\bigr)\left(h+\frac{v}{4}\right) &\text{if }y \notin \displaystyle\bigcup_{i=1}^{N}V_i\quad{\rm and}\; \vert g(y)\vert \geqslant h.
  \end{cases} 
\end{aligned}\]
The function $l$ is indeed continuous since the $V_i$ are clopen and $\bigcup_{i=1}^N V_i\supset K'$.\\
For each $i \in \{1,\dots,N\}$ and $y \in K\setminus (K' \cup \{x\})$ such that $y \in V_i$, we have: 
\[
   0<\chi(y)<\frac12 \implies h+\frac{v}{4}<h+\frac{v}{2}\Bigl(1-\chi(y)\Bigr)<h+\frac{v}{2}\cdot
\]

\noindent 
So we see that $\vert l(y)\vert < h + \frac{v}{2}$ for every $ y \in K \minus (K'\cup \{x\})$. Since $\vert l(z)\vert =h+\frac{v}{2}$ for all $z\in K'$, Corollary~\ref{l11} then implies that $\Phi ^{-1}(l)(y)=l(y)$ for all $y \in  K \minus (K'\cup \{x\})$. Since $K\setminus K'$ is dense in $K$, it follows that $\Phi ^{-1}(l)(y)=l(y)$ for all $y \neq x$. Moreover $l(x)=1$, so Lemma~\ref{l5} implies that $\Phi ^{-1}(l)(x)=1$. Thus, we have proved that
\[ \Phi^{-1}(l)=l.\]

\smallskip\noindent
Now, we are going to show that \[ \Vert \Phi^{-1}(g)-l\Vert <v,\] from which we will easily get a contradiction. Denote $\Phi^{-1}(g)$ by $g'$.\\
First, recall that we are assuming that $g'(x) >h$, so $\vert g'(x)-l(x)\vert =1-g'(x)<1-h=v$. \\
Next, let us show that $\vert g'(y)-l(y)\vert\leqslant\frac34 v$ for every $y \in K \minus (K'\cup \{x\})$. \\
Let $y \in K \minus (K'\cup \{x\})$. \\ 
- If $\vert g(y)\vert <h$, then Corollary~\ref{l11} implies that $g'(y)=g(y)$. Hence, $\vert g'(y)-l(y)\vert =0$.\\
- Assume that $y \in V_i$ for some $i$ and $g_{|_{V_i}}\equiv h$. Then $g(y)=h$, and Lemma~\ref{l5} implies that $g'(y) \in [h, 1]$. Moreover, $h+\frac{v}{4}< l(y)< h+\frac{v}{2}\cdot$ So, we see that if $g'(y)\leqslant l(y)$, then $\vert g'(y)-l(y)\vert =l(y)-g'(y)\leqslant h + \frac{v}{2}-h = \frac{v}{2};$ whereas 
if $g'(y)>l(y)$, then $\vert g'(y)-l(y)\vert =g'(y)-l(y)\leqslant 1-h- \frac{v}{4}= \frac{3}{4}v.$ \\
- Similarly, one gets $\vert g'(y)-l(y)\vert\leqslant\frac34 v$ when $y \in V_i$ and $g_{|_{V_i}}\equiv-h$.\\
- Assume that $y \notin \bigcup _{i=1}^{N}V_i$ and $g(y) \leqslant -h$. 
Then, $g'(y) \in [-1,-h]$ (from Lemma~\ref{l5}) and $l(y)=-h-\frac{v}{4}\cdot$ So, if $g'(y) \geqslant l(y)$ then $\vert g'(y)-l(y)\vert =g'(y)-l(y)\leqslant -h +h + \frac{v}{4}=\frac{v}{4};$ and if 
 $g'(y)<l(y)$ then $\vert g'(y)-l(y)\vert = l(y)-g'(y) \leqslant-h-\frac{v}{4}+1 = \frac{3}{4}v.$ \\
- Finally, one gets in the same way that $\vert g'(y)-l(y)\vert\leqslant \frac34 v$ when $y \notin \bigcup _{i=1}^{N}V_i$ and $g(y) \geqslant h$.

\smallskip\noindent
Hence, we have shown that $\vert g'(y)-l(y)\vert \leqslant \frac{3}{4}v$ for every $ y \in K \minus (K'\cup \{x\})$. Since $K\setminus K'$ is dense in $K$, 
it follows that $\vert g'(y)-l(y)\vert \leqslant \frac{3}{4}v$ for all $y \neq x$; and hence we get $\Vert g'-l\Vert <v$ since we observed above that $\vert g'(x)-l(x)\vert <v$.

\smallskip\noindent
It follows that
\[ v=1-h=l(x)-g(x) \leqslant \Vert l-g\Vert \leqslant \Vert \Phi^{-1}(l)-\Phi^{-1}(g)\Vert =\Vert l-\Phi^{-1}(g)\Vert <v,\]
which is a contradiction.

\smallskip\noindent
The case $h_{i_0} <0$ is analogous. 

\smallskip\noindent
So, we have shown that $\Phi ^{-1}(g)(x)= h_{i}$ for all $i \in \{1, \dots, N \}$ and every $x\in (K\setminus K')\cap V_i$. And since $K \minus K' $ is dense in $K$, it follows that $\Phi  ^{-1}(g)_{|_{V_i}}\equiv h_i$, for all $ i \in \{1,\dots,N\}$.
\end{proof}

\smallskip
It is now easy to conclude the proof of Lemma~\ref{l13}. On the one hand, since $\vert h_i\vert=h$ for all $i\in\{ 1,\dots ,N\}$, Lemma~\ref{l8} (together with the density of $K\setminus K'$) implies that $\Phi  (B_{\mathbf V,\mathbf h}) \subset B_{\mathbf V,\mathbf h}$. On the other hand, 
Fact~\ref{l12} implies that $\Phi  ^{-1}(B_{\mathbf V,\mathbf h}) \subset B_{\mathbf V,\mathbf h}$. So we have $\Phi  (B_{\mathbf V,\mathbf h}) = B_{\mathbf V, \mathbf h}$; and hence  $\Phi _{|_{B_{\mathbf V, \mathbf h}}}=Id$ by Fact~\ref{l10}.
\end{proof}

\subsection{Proof of the main lemma: $n=1$} In this subsection, we prove Lemma~\ref{l17} in the case $n=1$. (We need to treat this case separately because our proof of the inductive step from $n$ to $n+1$ works only if $n\geqslant 1$.) Let us recall the statement.

\begin{lemma} Let $f \in B$ be such that $u_{0,f} \notin \{0,1\}$. For any $x \in K\setminus K'$,
\begin{enumerate}[label={\rm \arabic*)}]
\item if $\vert f(x)\vert < u_{1,f}$, then $\Phi (f)(x)=f(x)$,
\item if $f(x) \geqslant u_{1,f}$, then $\Phi (f)(x) \in [u_{1,f},f(x)]$, 
\item if $f(x) \leqslant -u_{1,f}$, then $\Phi (f)(x) \in [f(x),-u_{1,f}]$. 
\end{enumerate}
\end{lemma}
\begin{proof} 1) Let us first assume that $f \in B_{\mathbf V,\mathbf h}$ for some $\mathbf V=(V_1,\dots ,V_N)\in\mathcal V$ and some  $\mathbf h=(h_1,\dots ,h_N) \in [-1,1]^N$ such that $h =\min \, \bigl\{\vert h_i \vert; \; i \in \{1,\dots, N\}\bigr\} \notin \{0,1\}$. 
Note that $h = u_{0,f}$. 

\smallskip\noindent
Let $x \in \ K\setminus K'$ be such that $\vert f(x)\vert  < u_{1,f} $. \\
If $\vert f(x)\vert \leqslant u_{0,f}$, then Lemma~\ref{l8} implies that $\Phi (f)(x)=f(x)$. So we may suppose 
that $\vert f(x)\vert >u_{0,f}$. \\
Let us define $g \in B$ as follows: 
\[\begin{aligned}
  g(y) = \begin{cases}
  \mathrm{sgn}\bigl(f(x)\bigr) &\text{if }y=x, \\
  \mathrm{sgn}(h_i)\,\vert f(x)\vert  &\text{if }y \in V_i\quad\hbox{for some $i$ and}\; y \neq x,\\
  f(y) &\text{if } y \notin \biggl(\displaystyle\bigcup_{i=1}^{N}V_i\cup \{x\}\biggr).
  \end{cases} 
\end{aligned}\]
Lemma~\ref{l13} implies that $\Phi (g)=g$. 

\smallskip\noindent
Let us show that $\Vert f-g\Vert =1-\vert f(x)\vert$. \\
First, we see that $\vert g(x)-f(x)\vert =1-\vert f(x)\vert $. \\
Now, let $y \in K \minus \{x\}$.  If $ y \notin \bigcup_{i=1}^{N}V_i $, then $\vert f(y)-g(y)\vert =0$. 
If $y \in V_i$ for some $i$, two cases may occur: $\vert f(x)\vert < \vert h_i\vert $ or $\vert f(x)\vert \geqslant \vert h_i\vert $. \\
- If $\vert f(x)\vert  < \vert h_i\vert $, then $\vert f(y)-g(y)\vert =\vert h_i\vert -\vert f(x)\vert \leqslant1-\vert f(x)\vert $.\\
- If $\vert f(x)\vert  \geqslant \vert h_i\vert$, then $\vert f(y)-g(y)\vert =\vert f(x)\vert -\vert h_i\vert \leqslant\vert f(x)\vert -u_{0,f}$. Moreover, since $\vert f(x)\vert < u_{1,f}= \frac{1+u_{0,f}}{2}$, we have $\vert f(x)\vert -u_{0,f}< 1-\vert f(x)\vert $. 
Hence, $\vert f(y)-g(y)\vert \leqslant 1-\vert f(x)\vert $.

\smallskip\noindent
So we have indeed shown that $\Vert f-g\Vert =1-\vert f(x)\vert $. Hence, \begin{align}\label{ineqphi1}
    \vert \Phi (f)(x)-\Phi (g)(x)\vert \leqslant \Vert \Phi (f)-\Phi (g)\Vert \leqslant \Vert f-g\Vert =1-\vert f(x)\vert . 
\end{align}
Since $\mathrm{sgn}\bigl(\Phi (f)(x)\bigr)= \mathrm{sgn}\bigl(f(x)\bigr)$ (from Lemma~\ref{l8}) and $\Phi (g)=g$,
\begin{center}
    $(\ref{ineqphi1}) \implies 1-\vert \Phi (f)(x)\vert \leqslant1-\vert f(x)\vert  \implies \vert \Phi (f)(x)\vert \geqslant\vert f(x)\vert $.
\end{center}
Moreover, Lemma~\ref{l5} implies that $\vert \Phi (f)(x)\vert \leqslant\vert f(x)\vert $. \\
So we have $\vert \Phi (f)(x)\vert =\vert f(x)\vert $ and $\mathrm{sgn}\bigl(\Phi (f)(x)\bigr)= \mathrm{sgn}\bigl(f(x)\bigr)$, \mbox{\it i.e.} $\Phi (f)(x)=f(x)$.

\medskip
Now, we can easily prove 1) in the general case, approximating $f$ by functions taking only finitely many values. 
Let us fix $x \in \ K\setminus K'$ such that $\vert f(x)\vert  < u_{1,f} $. \\
Since $K$ is zero-dimensional, for any $m\in\N$ one can find $g_m\in B$ taking only finitely many values such that $\Vert f-g_m\Vert \leqslant\frac{1}{m}$, $f(x)=g_m(x)$ and $u_{0,g_m}=u_{0,f}$. \\ Then  $u_{1,g_m}=u_{1,f}$ and hence $\vert g_m(x)\vert =\vert f(x)\vert < u_{1,g_m}$. So  $\Phi (g_m)(x)=g_m(x)=f(x)$, by the special case we have already treated (since $g_m$ takes only finitely many values, it belongs to $B_{\mathbf V,\mathbf h}$ for some pair $(\mathbf V,\mathbf h)$). \\
Hence, for any $ m \in \N$, $\vert \Phi (f)(x)-\Phi (g_m)(x)\vert \leqslant\Vert \Phi (f)- \Phi (g_m)\Vert \leqslant \Vert f-g_m\Vert \leqslant\frac{1}{m}\cdot$ It follows that for any $ m \in \N$, $\vert \Phi (f)(x)-f(x)\vert \leqslant \frac{1}{m}$; so we obtain $\Phi (f)(x)=f(x)$ by letting $m$ tend to infinity.

\medskip
2) Since $f(x) \geqslant   u_{1,f}>0$, Lemma~\ref{l5} implies that $\Phi (f)(x) \leqslant f(x)$. So it remains to show that $\Phi (f)(x) \geqslant  u_{1,f}$.\\
Let $\varepsilon >0$ be arbitrary, and let us choose  clopen sets $V_1,\dots ,V_N$ as in the proof of Lemma~\ref{l8}; that is, 
\begin{itemize}
\item[-] $V_i\cap K'\neq\emptyset$ for all $i$ and $\bigcup_{i=1}^N V_i\supset K'$,
\item[-] the $V_i$ are pairwise disjoint and do not contain $x$,
\item[-] $f$ has constant sign $\varepsilon_i$ on each $V_i$,
\item[-] ${\rm diam}\bigl(f(V_i)\bigr)<\varepsilon$ for $i=1,\dots ,N$.
\end{itemize}

\noindent
We define $g \in B$ as follows: 
\[ \begin{aligned}
g(y) = \begin{cases}
1 &\text{if }y=x, \\
\varepsilon_i u_{1,f} &\text{if } y \in V_i\quad\hbox{for some}\; i \in \{1,\dots,N\},\\
f(y) &\text{if } y \notin \biggl(\displaystyle\bigcup _{i=1}^{N}V_i\cup \{x\}\biggr).
\end{cases}
\end{aligned}\]

\noindent
Lemma~\ref{l13} implies that $\Phi (g)=g$. 

\smallskip\noindent
Let us show that $\Vert f-g\Vert \leqslant 1-u_{1,f}+\varepsilon$. 

\smallskip\noindent
First, we have $\vert f(x)-g(x)\vert =1-f(x)\leqslant 1-u_{1,f}$. \\
Now, let $y \in K \minus \{x\}$. \\
If $y \notin \bigcup _{i=1}^{N}V_i$, then $\vert f(y)-g(y)\vert =0$. \\
If $y \in V_i$ for some $i$,  choose $z\in V_i\cap K'$. Then $\vert f(y)-f(z)\vert <\varepsilon$ and $\mathrm{sgn}\bigl(f(z)\bigr)=\varepsilon_i$.\\ Hence, \[
\begin{aligned}
\vert f(y)-g(y) \vert &\leqslant \vert f(y)-f(z)\vert + \vert f(z)-g(y)\vert \\&<\varepsilon + \vert f(z)-g(y)\vert .
\end{aligned}\]
It follows that if $\vert f(z)\vert  <u_{1,f}$, then $\vert f(z)-g(y)\vert = u_{1,f}-\vert f(z)\vert \leqslant u_{1,f}-u_{0,f} =1-u_{1,f}$; whereas if $\vert f(z)\vert  \geqslant u_{1,f}$, then $\vert f(z)-g(y)\vert =\vert f(z)\vert -u_{1,f}\leqslant 1-u_{1,f}$. In either case, we have $\vert f(y)-g(y)\vert < 1-u_{1,f}+ \varepsilon$. 

\smallskip\noindent So, we have shown that $\Vert f-g\Vert \leqslant  1-u_{1,f}+ \varepsilon$.\\
Since $\Phi $ is non-expansive, we obtain
\[\begin{aligned}
\Vert \Phi (f)-g\Vert =\Vert \Phi (f)-\Phi (g)\Vert  \leqslant1-u_{1,f}+\varepsilon &\implies \vert \Phi (f)(x)-g(x)\vert \leqslant1-u_{1,f} + \varepsilon \\&\implies 1-\Phi (f)(x) \leqslant1-u_{1,f} + \varepsilon \\&\implies \Phi (f)(x) \geqslant u_{1,f} -\varepsilon.
\end{aligned}\]
Since $\varepsilon>0$ is arbitrary, it follows that $\Phi (f)(x) \geqslant u_{1,f}$. 

\medskip
3) The proof is similar to that of 2).
\end{proof}

\subsection{Proof of the main lemma: general case} In this subsection, we prove Lemma~\ref{l17} by induction on $n\in\Z_+$. 

\medskip
We already know that the result holds true for $n=0$ and $n=1$. Let us now fix $n\geqslant 1$. Assume that Lemma~\ref{l17} has been proved for $n$, and let us prove it for $n+1$.

\medskip\noindent
\begin{step1}  For each $g \in B$ such that $u_{0,g}\notin \{0,1\}$ and $x \in K\setminus K'$, if $\vert g(x)\vert < u_{n,g}$, then $\Phi ^{-1}(g)(x)=g(x)$.
\end{step1}
\begin{proof}
Let $g \in B$ be such that $u_{0,g}\notin \{0,1\}$, and let $x \in K\setminus K'$ be such that $\vert g(x)\vert < u_{n,g}$. \\
Denote $\Phi ^{-1}(g)$ by $f$.\\
Since $u_{0,f}=u_{0,g}$ (from Corollary~\ref{l9}), we have $u_{0,f} \notin \{0,1\}$ and $u_{n,f}=u_{n,g}$. \\
Three cases may occur: $\vert f(x)\vert <u_{n,f}$, $f(x)\geqslant u_{n,f}$ or $f(x) \leqslant- u_{n,f}$.\\
If $f(x) \geqslant u_{n,f}$, then the induction hypothesis implies that $g(x) \geqslant u_{n,f}=u_{n,g}$, which is a contradiction. Similarly, 
if $f(x) \leqslant -u_{n,f}$ then $g(x) \leqslant -u_{n,f}=-u_{n,g}$, which is again a contradiction. Hence, $\vert f(x)\vert  <u_{n,f}$. By the induction hypothesis, it follows that $f(x)=g(x)$, \textit{i.e.} $\Phi ^{-1}(g)(x)=g(x)$.
\end{proof}

\smallskip
\begin{step2}  Let $\mathbf V=(V_1,\dots ,V_N)\in\mathcal V$ and let $\mathbf h=(h_1, \dots,h_N) \in [-1,1]^N$. Assume that $h= \min \, \bigl\{\vert h_i\vert  ;\;  i \in \{1,\dots,N\}\bigr\} \notin \{0,1\}$, and that $\vert h_i\vert \leqslant\frac{n+h}{n+1}$ for all $ i \in \{1,\dots,N\}$.  
Then, $\Phi ^{-1}(B_{\mathbf V,\mathbf h}) \subset B_{\mathbf V, \mathbf h}$. 
\end{step2}
\begin{proof} Let $g \in B_{\mathbf V,\mathbf h}$. Our goal is to show that $\Phi ^{-1}(g)_{|_{V_i}}\equiv h_i$, for every $i \in \{1,\dots,N\}$. And since $K\setminus K'$ is dense in $K$, it is in fact enough to show that $\Phi^{-1}(g)(x)=h_i$ for all $i$ and every $x\in (K\setminus K')\cap V_i$.\\
Note that \[  u_{0,g}=h \qquad{\rm and}\qquad u_{n,g} = \frac{n+h}{n+1}\cdot\]
Let $i_0 \in \{1,\dots,N\}$ and let $x\in (K\setminus K')\cap V_{i_0}$. Then, $g(x)=h_{i_0}$. 
We want to show that $\Phi ^{-1}(g)(x)=h_{i_0}$. \\
If $\vert h_{i_0}\vert < \frac{n+h}{n+1}= u_{n,g}$, then Step~1 implies that $\Phi ^{-1}(g)(x)=g(x)=h_{i_0}$. \\
If $\vert h_{i_0}\vert = \frac{n+h}{n+1}$, suppose for example that $h_{i_0} >0$, \textit{i.e.} $g(x)=h_{i_0}= \frac{n+h}{n+1}\cdot$ \\
Since $g(x)>0$, Lemma~\ref{l5} implies that $\Phi ^{-1}(g)(x) \geqslant g(x) = \frac{n+h}{n+1}\cdot$ So, it remains to show that $\Phi ^{-1}(g)(x) \leqslant \frac{n+h}{n+1}\cdot$ \\
For the sake of contradiction, suppose that $\Phi ^{-1}(g)(x) > \frac{n+h}{n+1}\cdot$ \\
Let us set 
\[ u= \frac{n+h}{n+1}\qquad\hbox{and}\qquad v=1-u.\]
We note that \begin{align}\label{trucmachin}\frac{(2n+1)h+1}{2(n+1)}<u. \end{align}
Indeed, \[\frac{(2n+1)h+1}{2(n+1)}-u= \frac{(2n+1)h+1}{2(n+1)}- \frac{n+h}{n+1}= \frac{(2n-1)(h-1)}{2(n+1)}<0.\]
Moreover, since $h <1$, we have $h <\frac{(2n+1)h+1}{2(n+1)}\cdot$\\
Now, let us define $l \in B $ as follows: 
\[\begin{aligned}
l(y) = \begin{cases}
1 &\text{if }y=x, \\
g(y) &\text{if } \vert g(y)\vert <h\\&\text{or } \frac{(2n+1)h+1}{2(n+1)} \leqslant\vert g(y)\vert <u, \\
\mathrm{sgn}\bigl(g(y)\bigr) \frac{(2n+1)h+1}{2(n+1)} &\text{if }  h \leqslant \vert g(y)\vert < \frac{(2n+1)h+1}{2(n+1)}\comma\\
\mathrm{sgn}\bigl(g(y)\bigr)\Bigl(u+\frac{v}{3(n+1)}\Bigr) &\text{if }\vert g(y)\vert \geqslant u \;{\rm and}\; y \neq x.
\end{cases}
\end{aligned} \]

\smallskip\noindent
Note that $l$ is constant on $V_{i_0}\minus \{x\}$ and on each set $V_i$ for $i \neq i_0$, since $g \in B_{\mathbf V,\mathbf h}$. If we choose $z_i\in V_i \minus \{x\}$ for $i=1,\dots ,N$, then $l\in B_{\mathbf V', \mathbf h'}$ where $\mathbf V'=(V'_1, \dots , V'_N)$, $V'_i=V_i$ for all $i \neq i_0$, $V'_{i_0}=V_{i_0}\minus \{x\}$ and $\mathbf h'=(l(z_1),\dots ,l(z_N))$.\\
Note also that $u_{0,l}= \frac{(2n+1)h+1}{2(n+1)}\cdot$ 
Hence, $u_{n,l}=\frac{n+u_{0,l}}{n+1}= \frac{2n^2+2n+2nh+h+1}{2(n+1)^2}\cdot$\\
Moreover, $u+\frac{v}{2(n+1)}= u +\frac{1-u}{2(n+1)}= \frac{2n^2+2n+2nh+h+1}{2(n+1)^2}\cdot$ So, we get 
\[ u_{n,l}=u +\frac{v}{2(n+1)}\cdot\]
Now, we have $\vert l(y)\vert < u_{n,l}$, for all $y\neq x$. 
Then, Step~1 implies that $\Phi ^{-1}(l)(y)=l(y)$, for all $ y \in K \minus (K'\cup \{x\})$. And since $K\setminus K'$ is dense in $K$, it follows that $\Phi ^{-1}(l)(y)=l(y)$ for all $ y \neq x$. \\
Moreover, since $l(x)=1$, Lemma~\ref{l5} implies that $\Phi ^{-1}(l)(x)=1$. So we have shown that
\[ \Phi ^{-1}(l)=l.\]

\smallskip\noindent
Now, we are going to show that $\Vert \Phi^{-1}(g)-l\Vert <v$, from which we will obtain a contradiction. Denote $\Phi^{-1}(g)$ by $g'$.\\
First, note that according to our assumption, we have $g'(x) >u$; so $\vert g'(x)-l(x)\vert =1-g'(x)<1-u=v$. \\
Now, let $y \in K \minus (K'\cup \{x\})$. \\
- If $\vert g(y)\vert <h$ or $\frac{(2n+1)h+1}{2(n+1)} \leqslant\vert g(y)\vert <u$, then $\vert g(y)\vert <u=u_{n,g}$ and Step~1 implies that $g'(y)=g(y)$. Hence, $\vert g'(y)-l(y)\vert =0$.\\
- If $h \leqslant\vert g(y)\vert < \frac{(2n+1)h+1}{2(n+1)}$ then (\ref{trucmachin}) implies that $\vert g(y)\vert <u=u_{n,g}$. 
Hence, Step~1 implies that $g'(y)=g(y)$. So we get
\[ \begin{aligned}
\vert l(y)-g'(y)\vert &= \vert l(y)-g(y)\vert \\&= \frac{(2n+1)h+1}{2(n+1)}-\vert g(y)\vert \\&\leqslant \frac{(2n+1)h+1}{2(n+1)}-h \\&=\frac{1-h}{2(n+1)}\cdot
\end{aligned}\]
Since $v=1-\frac{n+h}{n+1}=\frac{1-h}{n+1}$, it follows that $\vert l(y)-g'(y)\vert \leqslant \frac{v}{2}\cdot$ \\
- If $g(y)\geqslant u $, then Lemma~\ref{l5} implies that $g'(y) \in [u, 1]$ and $l(y)=u+\frac{v}{3(n+1)}\cdot$ 
So, 
if $g'(y)\leqslant l(y)$, then we obtain $\vert g'(y)-l(y)\vert =l(y)-g'(y)\leqslant u + \frac{v}{3(n+1)}-u = \frac{v}{3(n+1)}$; whereas if 
$g'(y)>l(y)$, then $\vert g'(y)-l(y)\vert =g'(y)-l(y)\leqslant 1-u- \frac{v}{3(n+1)}= \frac{3n+2}{3(n+1)}v.$ In either case, we see that $\vert g'(y)-l(y)\vert\leqslant \frac{3n+2}{3(n+1)}v.$\\
- Finally, one shows in the same way that if $g(y) \leqslant -u$ then $\vert g'(y)-l(y)\vert\leqslant \frac{3n+2}{3(n+1)}v\cdot$

\smallskip\noindent
Hence,  we have $\vert g'(y)-l(y)\vert \leqslant \frac{3n+2}{3(n+1)}v$ for all $ y \in K \minus (K'\cup \{x\})$; and it 
 follows that $\vert g'(y)-l(y)\vert \leqslant \frac{3n+2}{3(n+1)}v$ for all $ y \neq x$. Since $\vert g'(x)-l(x)\vert <v$ (as observed above), we conclude that $\Vert g'-l\Vert <v$.

\smallskip\noindent Therefore, \begin{center}
    $v=1-u=l(x)-g(x) \leqslant \Vert l-g\Vert \leqslant \Vert \Phi^{-1}(l)-\Phi^{-1}(g)\Vert =\Vert l-\Phi^{-1}(g)\Vert <v$,
\end{center}
which is a contradiction.

\smallskip\noindent
 The case $h_{i_0} <0$ is analogous.
\end{proof} 

\smallskip 
\begin{step3}  Let $\mathbf V=(V_1,\dots ,V_N)\in\mathcal V$ and let $\mathbf h =(h_1, \dots,h_N) \in [-1,1]^N$. Assume that $h= \min \,\bigl\{\vert h_i\vert  ;\;  i \in \{1,\dots,N\}\bigr\} \notin \{0,1\}$, and that $\vert h_i\vert \leqslant\frac{n+h}{n+1}$ for all $ i \in \{1,\dots,N\}$.  
Then, $\Phi _{|_{B_{\mathbf V, \mathbf h}}}=Id$.
\end{step3}

\begin{proof} The induction hypothesis (together with the density of $K\setminus K'$) implies that $\Phi  (B_{\mathbf V,\mathbf h}) \subset B_{\mathbf V,\mathbf h}$. Moreover, 
Step~2 implies that $\Phi  ^{-1}(B_{\mathbf V,\mathbf h}) \subset B_{\mathbf V,\mathbf h}$. So we have
 $\Phi  (B_{\mathbf V,\mathbf h}) = B_{\mathbf V, \mathbf h}$; and hence Fact~\ref{l10} implies that $\Phi _{|_{B_{\mathbf V,\mathbf h}}}=Id$.
\end{proof}

\smallskip
\begin{step4} For each $f \in B$ such that $u_{0,f} \notin \{0,1\}$ and $x \in K\setminus K'$, if $\vert f(x)\vert < u_{(n+1),f} $, then $\Phi (f)(x)=f(x)$. 
\end{step4}
\begin{proof}  Let us first assume that $f \in B_{\mathbf V,\mathbf h}$ for some $\mathbf V=(V_1,\dots ,V_N)\in\mathcal V$ and some  $\mathbf h =(h_1, \dots, h_N) \in [-1,1]^N$ such that $h =\min \, \bigl\{\vert h_i \vert; \; i \in \{1,\dots, N\}\bigr\} \notin \{0,1\}$. \\
Note that $h = u_{0,f}$. \\
Let $x \in \ K\setminus K'$ be such that $\vert f(x)\vert  < u_{(n+1),f} $. \\
If $\vert f(x)\vert  \leqslant u_{n,f}$, then the induction hypothesis implies that $\Phi (f)(x)=f(x)$. 

\smallskip\noindent
So we consider the case where $\vert f(x)\vert >u_{n,f}$. \\
Since $\vert f(x)\vert >u_{n,f}\geqslant u_{1,f}$ (because $n \geqslant 1$) and since   $u_{1,f}= \frac{1+ u_{0,f}}{2} $,  it follows that $u_{0,f}+1-\vert f(x)\vert <\vert f(x)\vert $. \\
Now, we define $g \in B$ as follows: 
\[\begin{aligned}
g(y)= \begin{cases}
\mathrm{sgn}\bigl(f(x)\bigr) &\text{if }y=x, \\
f(y)  &\text{if }\vert f(y)\vert <u_{0,f} \\&\text{or } u_{0,f}+1-\vert f(x)\vert \leqslant\vert f(y)\vert  <\vert f(x)\vert ,\\
\mathrm{sgn}\bigl(f(y)\bigr)(u_{0,f}+1-\vert f(x)\vert ) &\text{if } u_{0,f} \leqslant \vert f(y)\vert  < u_{0,f}+1-\vert f(x)\vert ,\\
\mathrm{sgn}\bigl(f(y)\bigr)\vert f(x)\vert  &\text{if }\vert f(y)\vert \geqslant \vert f(x)\vert \text{ and } y \neq x. 
\end{cases}
\end{aligned}\]
Note that $g$ is constant on each set $V_i \minus \{x\}$ since $f \in B_{\mathbf V,\mathbf h}$. If we choose $z_i\in V_i \minus \{x\}$ for $i=1,\dots ,N$, then $g\in B_{\mathbf V', \mathbf h'}$ where $\mathbf V'=(V_1\minus \{x\}, \dots, V_N\minus \{x\})$ and $\mathbf h'=(g(z_1),\dots ,g(z_N))$. Note also that $\vert g(y)\vert \leqslant \vert f(x)\vert$ for all $y\neq x$.\\
We can see that $u_{0,g}=u_{0,f}+1-\vert f(x)\vert $. 
Hence, $u_{n,g}=\frac{n+u_{0,f}+1-\vert f(x)\vert }{n+1}\cdot$ \\
Next, we observe that \[ u_{n,g}>\vert f(x)\vert .\]

\noindent Indeed, since $\vert f(x)\vert < u_{(n+1),f}=\frac{n+1+u_{0,f}}{n+2}$, we have $(n+2)\vert f(x)\vert < n+1+u_{0,f},$ so we get
\begin{align*}
u_{n,g}-\vert f(x)\vert  =\frac{n+u_{0,f}+1-(n+2)\vert f(x)\vert }{n+1}> 0. 
\end{align*}
Since $g \in B_{\mathbf V', \mathbf h'}$, $u_{0,g}\notin \{0,1\}$ and $\vert g(z_i)\vert \leqslant \vert f(x)\vert < u_{n,g}$, for all $i \in \{1,\dots,N\}$, Step~3 implies that \[ \Phi (g)=g.\]

\smallskip\noindent Let us show that $\Vert g-f\Vert = 1-\vert f(x)\vert$.\\
First, we have $\vert g(x)-f(x)\vert =1-\vert f(x)\vert $. \\
Now, let  $y \in K \minus \{ x\}$. \\
- If $\vert f(y)\vert  < u_{0,f}$ or $u_{0,f}+1-\vert f(x)\vert \leqslant \vert f(y)\vert <\vert f(x)\vert $, then $\vert f(y)-g(y)\vert =0$. \\
- If $u_{0,f} \leqslant \vert f(y)\vert <u_{0,f}+1-\vert f(x)\vert $, then 
\begin{align*}
    \vert f(y)-g(y)\vert &= \vert g(y)\vert -\vert f(y)\vert \\&\leqslant u_{0,f}+1-\vert f(x)\vert -u_{0,f}\\&=1-\vert f(x)\vert .
\end{align*}
- If $\vert f(y)\vert \geqslant\vert f(x)\vert $, then 
\begin{align*}
    \vert f(y)-g(y)\vert  &= \vert f(y)\vert -\vert g(y)\vert \\&=\vert f(y)\vert -\vert f(x)\vert \\&\leqslant 1-\vert f(x)\vert . 
\end{align*}
So we have proved that $\Vert f-g\Vert =1-\vert f(x)\vert $. Hence, \begin{align}\label{encorephi}
    \vert \Phi (f)(x)-\Phi (g)(x)\vert \leqslant \Vert \Phi (f)-\Phi (g)\Vert \leqslant \Vert f-g\Vert =1-\vert f(x)\vert . 
\end{align}
Since $\mathrm{sgn}\bigl(\Phi (f)(x)\bigr)= \mathrm{sgn}\bigl(f(x)\bigr)$ (from Lemma~\ref{l8}) and $\Phi (g)=g$,
\begin{center}
    $(\ref{encorephi}) \implies 1-\vert \Phi (f)(x)\vert \leqslant1-\vert f(x)\vert  \implies \vert \Phi (f)(x)\vert \geqslant\vert f(x)\vert $.
\end{center}
Moreover, Lemma~\ref{l5} implies that $\vert \Phi (f)(x)\vert \leqslant\vert f(x)\vert $. \\
So we have $\vert \Phi (f)(x)\vert =\vert f(x)\vert $ and $\mathrm{sgn}\bigl(\Phi (f)(x)\bigr)= \mathrm{sgn}\bigl(f(x)\bigr)$, \mbox{\it i.e.} $\Phi (f)(x)=f(x)$.

\smallskip
Now, we can prove Step~4 by an approximation argument. \\
Let $f \in B$ be such that $u_{0,f}\notin\{0,1\}$ and let $x \in \ K\setminus K'$ be such that $\vert f(x)\vert  < u_{(n+1),f} $. \\
For any $m \in \N$, one can find $g_m \in B$  taking finitely many values such that $\Vert f-g_m\Vert \leqslant\frac{1}{m}$, $f(x)=g_m(x)$ and $u_{0,g_m}=u_{0,f}$. Since $u_{0,g_m}=u_{0,f}$, it follows that  $u_{(n+1),g_m}=u_{(n+1),f}$.  So 
we have $\vert g_m(x)\vert =\vert f(x)\vert < u_{(n+1),g_m}$, hence  $\Phi (g_m)(x)=g_m(x)=f(x)$ (we have used the fact that $g_m\in B_{\mathbf V,\mathbf h}$ for some pair $(\mathbf V,\mathbf h)$). Consequently, $\vert \Phi (f)(x)-\Phi (g_m)(x)\vert \leqslant\Vert \Phi (f)- \Phi (g_m)\Vert \leqslant \Vert f-g_m\Vert \leqslant\frac{1}{m}$ for all $ m \in \N$. It follows that $\vert \Phi (f)(x)-f(x)\vert \leqslant \frac{1}{m}$ for all $ m \in \N$; so we obtain $\Phi(f)(x)=f(x)$ by letting $m$ tend to infinity. 
\end{proof}

\smallskip
\begin{step5} For each $f \in B$ such that $u_{0,f} \notin \{0,1\}$ and $x \in K\setminus K'$, if $f(x)\geqslant u_{(n+1),f} $, then $\Phi (f)(x)\in [u_{(n+1),f},f(x)]$.
\end{step5}
\begin{proof} 
Let us first assume that $f \in B_{\mathbf V,\mathbf h}$ for some $\mathbf V=(V_1,\dots, V_N)\in\mathcal V$ and some  $\mathbf h=(h_1,\dots, h_N)\in [-1,1]^N$ such that $h =\min \, \bigl\{\vert h_i \vert; \; i \in \{1,\dots, N\}\bigr\} \notin \{0,1\}$. 
Note that $h = u_{0,f}$. \\
Let $x \in \ K\setminus K'$ be such that $ f(x) \geqslant u_{(n+1),f}.$ \\
Since $u_{(n+1),f} > u_{1,f}$ (because $n \geqslant 1$), we have $u_{0,f}+1-u_{(n+1),f}<u_{(n+1),f} $. \\
Indeed, 
\begin{align*}
    u_{1,f}= \frac{1+u_{0,f}}{2} &\implies 1+ u_{0,f}=2 u_{1,f}<2u_{(n+1),f} \\&\implies 1+u_{0,f}-u_{(n+1),f}<u_{(n+1),f}. 
\end{align*}
Now, we define $g \in B$ as follows: 
\begin{align*}
    g(y)= \begin{cases}
    1 &\text{if }y=x, \\
    f(y) &\text{if } \vert f(y)\vert <u_{0,f}\\&\text{or } u_{0,f} +1-u_{(n+1),f}<\vert f(y)\vert <u_{(n+1),f}, \\
    \mathrm{sgn}\bigl(f(y)\bigr)\bigl(u_{0,f} +1-u_{(n+1),f}\bigr) &\text{if }u_{0,f} \leqslant \vert f(y)\vert  \leqslant u_{0,f}+1-u_{(n+1),f}, \\
    \mathrm{sgn}\bigl(f(y)\bigr)u_{(n+1),f} &\text{if }\vert f(y)\vert \geqslant u_{(n+1),f},\text{ and } y \neq x. 
    \end{cases}
\end{align*}

\noindent If we choose $z_i\in V_i \minus \{x\}$ for $i=1,\dots ,N$, then $g\in B_{\mathbf V', \mathbf h'}$ where $\mathbf V'=(V_1\minus \{x\}, \dots, V_N\minus \{x\})$ and $\mathbf h'=(g(z_1),\dots ,g(z_N))$.

\smallskip\noindent
Since $u_{0,g}=u_{0,f}+1-u_{(n+1),f}$, it follows that $u_{n,g}= \frac{n+u_{0,f}+1-u_{(n+1),f}}{n+1}\cdot$  And since $u_{(n+1),f}=\frac{n+1+u_{0,f}}{n+2} $, we get
\[ u_{n,g}=u_{(n+1),f}.\] 

\noindent
Hence, $g \in B_{\mathbf V',\mathbf h'}$, $u_{0,g}\notin \{0,1\}$ and $\vert g(z_i)\vert \leqslant u_{(n+1),f}=u_{n,g}$, for all $i \in \{1,\dots, N\}$. Step~3 then implies that \[\Phi (g)=g.\]

\smallskip\noindent Let us show that $\Vert f-g\Vert \leqslant 1-u_{(n+1),f}$. \\
First, we  have $\vert f(x)-g(x)\vert =1-f(x)\leqslant 1-u_{(n+1),f}$. \\
Now, let $y \in K\minus \{x\}$.\\
- If $\vert f(y)\vert <u_{0,f}$ or $u_{0,f} +1-u_{(n+1),f}<\vert f(y)\vert <u_{(n+1),f}$, then $\vert f(y)-g(y)\vert =0$. \\
- If $u_{0,f}\leqslant \vert f(y)\vert \leqslant u_{0,f}+1-u_{(n+1),f}$, then 
\begin{align*}
    \vert f(y)-g(y)\vert  &= \vert g(y)\vert -\vert f(y)\vert \\&\leqslant u_{0,f}+1-u_{(n+1),f}-u_{0,f} \\&=1-u_{(n+1),f}. 
\end{align*}
- If $\vert f(y)\vert \geqslant u_{(n+1),f}$, then 
\begin{align*}
    \vert f(y)-g(y)\vert &= \vert f(y)\vert -\vert g(y)\vert  \\&= \vert f(y)\vert -u_{(n+1),f}\\&\leqslant 1- u_{(n+1),f}. \end{align*}
So we have shown that $\Vert f-g\Vert \leqslant 1- u_{(n+1),f}$. Hence,
\begin{align}\label{ineqphi2}
    \vert \Phi (f)(x)-\Phi (g)(x)\vert \leqslant \Vert \Phi (f)-\Phi (g)\Vert \leqslant \Vert f-g\Vert \leqslant 1-u_{(n+1),f}. 
\end{align}
Since $\Phi (g)=g$,  
\[ (\ref{ineqphi2}) \implies 1-\Phi (f)(x)\leqslant1- u_{(n+1),f} \implies \Phi (f)(x) \geqslant u_{(n+1),f}.\]

\smallskip
Now, we can prove Step~5, again by an approximation argument. \\
Let $f \in B$ be such that $u_{0,f}\notin\{0,1\}$ and let $x \in \ K\setminus K'$ be such that $f(x) \geqslant u_{(n+1),f} $. \\
Since $f(x) \geqslant u_{(n+1),f} >0$, it follows from Lemma~\ref{l5} that $\Phi (f)(x) \leqslant f(x)$. So, it remains to show that $\Phi (f)(x) \geqslant u_{(n+1),f}$.\\
For any $m \in \N$, one can find  $g_m \in B$ taking finitely many values such that $\Vert f-g_m\Vert \leqslant\frac{1}{m}$, $f(x)=g_m(x)$ and $u_{0,g_m}=u_{0,f}$. Then, we have  
$u_{(n+1),g_m}=u_{(n+1),f}$. So $g_m(x)=f(x)\geqslant u_{(n+1),g_m}$, and hence  $\Phi (g_m)(x)\geqslant u_{(n+1),g_m}=u_{(n+1),f}$ for all $m\in\N$. Since 
 $\vert \Phi (f)(x)-\Phi (g_m)(x)\vert \leqslant\Vert \Phi (f)- \Phi (g_m)\Vert \leqslant \Vert f-g_m\Vert \leqslant\frac{1}{m}$, it follows that $\Phi (f)(x)\geqslant u_{(n+1),f}$.
\end{proof}

\smallskip
\begin{step6} For each $f \in B$ such that $u_{0,f} \notin \{0,1\}$ and $x \in K\setminus K'$, if $f(x)\leqslant -u_{(n+1),f} $, then $\Phi (f)(x)\in [f(x),-u_{(n+1),f}]$.
\end{step6}
\begin{proof} 
The proof  is analogous to that of Step~5.
\end{proof}

\medskip At this point, we have proved Lemma~\ref{l17} by induction. So the proof of Theorem~\ref{main} is finally complete.

\smallskip
\section{Proof of Theorem \ref{main2}}\label{S4} In this section, $K$ is a zero-dimensional compact Hausdorff space such that $K\setminus K'$ is dense in $K$, and we denote $B_{C(K)}$ by $B$. We fix a non-expansive homeomorphism $F:B\to B$, and we want to show that $F$ is an isometry.

\smallskip Looking back at the proof of Theorem~\ref{main}, we see that we only need two things:
\begin{itemize}
\item[-] that Lemma~\ref{l1} holds as stated, without assuming that $K'$ is a finite set,
\item[-] that Fact~\ref{l12} holds as stated, without assuming that $K$ is metrizable.
\end{itemize}

\smallskip Concerning Lemma~\ref{l1}, the only place in the proof where we used that $K'$ is finite is the end of the proof of Step~4, to get that the map $\sigma$ is injective.

Let us now show that $\sigma$ is injective by using the continuity of $F^{-1}$. For this we need the following three facts (which are true without assuming that $F^{-1}$ is continuous). Recall that $\alpha = F^{-1}(\ind)$.
\begin{fact}\label{f1}
Let $f \in B$ be such that $f=\alpha \cdot g$ where $g(x) \geqslant 0$ for all $x \in K$. 
Then, $F(f)(x) \geqslant 0$ for all $x \in K$. Moreover, if $a \in K\minus K'$ is such that $f(a)=0$, then $F(f)(\sigma(a))=0$. 
\end{fact}
\begin{proof}
For every $x \in K$, we have $\lvert f(x)-\alpha(x)\rvert=\lvert \alpha(x)g(x)-\alpha(x)\rvert=\lvert 1-g(x)\rvert\leqslant 1$ because $g(x)\geqslant 0$, hence $\Vert f-\alpha \Vert \leqslant 1$. Since $F$ is non-expansive, it follows that 
\[\begin{aligned}
\Vert F(f)-F(\alpha)\Vert \leqslant 1 &\implies \Vert F(f) - \ind \Vert \leqslant 1 \\&\implies F(f)(x) \geqslant 0 \text{, for all } x \in K.
\end{aligned}\]
Now, let $a \in K\minus K'$ be such that $f(a)=0$. We want to show that $F(f)(\sigma(a))=0$. \\
We know that $F(f)(\sigma(a))\geqslant 0$, so it remains to show that $F(f)(\sigma(a))\leqslant 0$. \\
Let $h$ be the following extreme point of $B$:
\[
h(x)= \begin{cases} 1 &\text{ if } x \neq \sigma(a),\\
-1 &\text{ if } x = \sigma(a).
\end{cases}
\]
Since $h$ is an extreme point of $B$, Fact~\ref{keyformulabiz} implies that \begin{center}
    $F^{-1}(h)(a)=h(\sigma(a))\alpha(a)=-\alpha(a)$,
\end{center}
and that for every $x \neq a$, we have
\begin{center}
    $F^{-1}(h)(x)=h(\sigma(x))\alpha(x)=\alpha(x)$,
\end{center}
since $\sigma(x) \neq \sigma(a)$ (because $\sigma_0$ is injective and $\sigma(K') \subset K'$). \\
Hence, \[\begin{aligned}
\Vert F^{-1}(h)-f \Vert \leqslant 1 &\implies \Vert h-F(f)\Vert \leqslant 1 \\&\implies \vert h(\sigma(a))-F(f)(\sigma(a))\vert \leqslant 1 \\&\implies 1+ F(f)(\sigma(a))\leqslant1 \\&\implies F(f)(\sigma(a))\leqslant0.
\end{aligned} \]
\end{proof}
\begin{fact}\label{f2}
Let $f \in B$ be such that $f=\alpha \cdot g$ where $g(x) \geqslant 0$ for all $x \in K$. 
Then, for every $a \in K \minus K'$, we have
\begin{center}
   $ 0\leqslant F(f)(\sigma(a))\leqslant g(a)$. 
\end{center}
\end{fact}
\begin{proof}
Let $a \in K \minus K'$. \\
If $f(a)=0$, then Fact~\ref{f1} implies that $F(f)(\sigma(a))=0$. \\
Now, assume that $f(a) \neq 0$. We know from Fact~\ref{f1} that $F(f)(\sigma(a))\geqslant 0$, so it remains to show that $F(f)(\sigma(a))\leqslant g(a)$. \\
We take $h$ as in the proof of Fact~\ref{f1}; hence $F^{-1}(h)(x)=\alpha(x)$ for all $x \neq a$, and  $F^{-1}(h)(a)=-\alpha(a)$. So, we obtain that $\vert F^{-1}(h)(x)-f(x)\vert \leqslant 1$ for all $x \neq a $, and $\vert F^{-1}(h)(a)-f(a)\vert=1+g(a)$. 
Hence,
\[\begin{aligned}
\Vert F^{-1}(h)-f\Vert =1+g(a) &\implies \Vert h -F(f) \Vert \leqslant 1+g(a) \\&\implies \vert h(\sigma(a))-F(f)(\sigma(a))\vert \leqslant 1+ g(a) \\&\implies 1+F(f)(\sigma(a))\leqslant 1+g(a)\\&\implies F(f)(\sigma(a)) \leqslant g(a).
\end{aligned} \]
\end{proof}
\begin{fact}\label{f3}
Let $U$ be a clopen subset of $K$ and let $f \in B$ be such that $f=\alpha \cdot g$ where 
\[
g(x)= \begin{cases} a &\text{ if } x \in U,\\
b &\text{otherwise} , 
\end{cases}
\]
and $0\leqslant b\leqslant a \leqslant1$. 
\begin{itemize}
    \item[\rm (i)] If $a=b$, then $F(f)=a \ind$.
    \item[\rm (ii)] If $a > b$, then 
    \begin{itemize}
        \item[-] for every $x \in K \minus U$, $F(f)(\sigma(x))=b$,
        \item[-] for every $x \in U$,  $b \leqslant F(f)(\sigma(x))\leqslant a $.
    \end{itemize}
\end{itemize}
\end{fact}
\begin{proof}
(i)  By Theorem~\ref{th1} (item~3), we have $F^{-1}(a\ind)=aF^{-1}(\ind)=a \alpha=f$. Hence, $F(f)=a\ind$. 

\smallskip\noindent
(ii) Define $l=c \alpha \in B$ where $c=\frac{a+b}{2}\cdot$ By (i), we have $F(l)=c \ind$. \\
For every $x \in U$, 
\begin{center}
    $\vert l(x)-f(x)\vert =\vert c\alpha(x)-a\alpha(x)\vert =a-c$,
\end{center}
and for every $x \in K \minus U$, 
\begin{center}
    $\vert l(x)-f(x)\vert =\vert c\alpha(x)-b\alpha(x)\vert =c-b$. 
\end{center}
Hence, $\Vert l-f\Vert =c-b=a-c$ because $c=\frac{a+b}{2}\cdot$ Since $F$ is non-expansive, we get $\Vert F(l)-F(f)\Vert \leqslant c-b$. \\
For every $x \in K $, we have 
\[\begin{aligned}
\vert F(l)(\sigma(x))-F(f)(\sigma(x))\vert \leqslant c-b &\implies F(l)(\sigma(x))-F(f)(\sigma(x)) \leqslant c-b \\ &\implies c-F(f)(\sigma(x)) \leqslant c-b \\&\implies F(f)(\sigma(x)) \geqslant b.
\end{aligned} \]
Hence, by Fact~\ref{f2}, we get $F(f)(\sigma(x))=b$ for all $x \in K\minus(K'\cup U)$, and $b \leqslant F(f)(\sigma(x))\leqslant a$ for all $x \in ( K \minus K') \cap U$. \\
Since $U$ is a clopen subset of $K$, $K \minus K'$ is dense in $K$ and $\sigma$ and $F(f)$ are continuous, the proof of (ii) is complete. 
\end{proof}
Now we are ready to show that $\sigma$ is injective. Towards a contradiction, suppose that there exist $u,v \in K$ such that $u \neq v$ and $\sigma(u)=\sigma(v):=z$. \\
Since $\sigma(K') \subset K'$ and $\sigma_0$ is injective, we get $u,v \in K'$. \\
Since $K$ is zero-dimensional, one can find a clopen neighbourhood $U$ of $u$ such that $v \notin U$. \\
Now, let us define $f \in B$ as follows: 
\[
f(x)= \begin{cases} a \alpha(x)&\text{ if } x \in U,\\
b \alpha(x)&\text{otherwise} , 
\end{cases}
\]
where $0\leqslant b<a\leqslant1$.\\
Fact~\ref{f3} implies that $F(f)(\sigma(x))=b$ for all $x \in K \minus U$, and $b \leqslant F(f)(\sigma(x)) \leqslant a$ for all $x \in U$.\\ 
In particular, we have $F(f)(z)=F(f)(\sigma(v))=b$ since $v \in K \minus U$. \\
Let $(a_d)$ be a net in $U\minus K'$ such that $a_d \longrightarrow u$. Since $F(f)$ and $\sigma$ are continuous, we see that $ F(f)(\sigma(a_d)) \longrightarrow F(f)(z)=b$, and Fact~\ref{f3} implies that for every $d$, we have $b \leqslant F(f)(\sigma(a_d))\leqslant a$. \\
Now, for every $d$, we define $g_d \in B$ as follows: 
\[
g_d(x)= \begin{cases} c \alpha(x)&\text{ if } x \neq a_d,\\
a \alpha(x)&\text{ if } x=a_d, 
\end{cases}
\]
where $c = \frac{a+b}{2}\cdot$ \\
Since $\{a_d\}$ is a clopen subset of $K$ (because $a_d \in K \minus K'$), Fact~\ref{f3} implies that $F(g_d)(\sigma(x)) =c$ for every $x \neq a_d$, and $c \leqslant F(g_d)(\sigma(a_d))\leqslant a$. This implies that $F(g_d)(y)=c$, for all $y \neq \sigma(a_d)$. Indeed, let $y \in K \minus (K'\cup \{\sigma(a_d)\})$. Since $\sigma_0$ is bijective, there exists $x \in K \minus (K'\cup \{a_d\})$ such that $\sigma(x)=y$. Hence, $F(g_d)(y)=c$. It follows that $F(g_d)(y)=c$, for all $y \neq \sigma(a_d)$, since $K \minus K'$ is dense in $K$.\\
For every $d$, we have 
\[
\vert g_d(x)-f(x)\vert = \begin{cases} \vert a \alpha(a_d)-a \alpha(a_d)\vert =0 &\text{ if } x = a_d,\\
\vert c \alpha(x) -a\alpha (x) \vert = a-c &\text{ if } x\in U \minus\{a_d\}, \\
\vert c \alpha(x)-b \alpha(x)\vert = c-b &\text{ if } x\in K\minus U, 
\end{cases}
\]
because $a_d \in U$. \\
Since $a-c=c-b$, it follows that $\Vert g_d -f \Vert =c-b$, for every $d$. Hence, $\Vert F(g_d) -F(f) \Vert \leqslant c-b$ and in particular $F(g_d)(\sigma(a_d))-F(f)(\sigma(a_d))  \leqslant c-b$, for every $d$. \\
Since we know that $ F(f)(\sigma(a_d)) \longrightarrow b$ and $F(g_d)(\sigma(a_d)) \geqslant c$ for every $d$, it follows that $F(g_d)(\sigma(a_d)) \longrightarrow c$. \\
Now, let us define $l \in B$ such that $l=c \alpha$. Fact~\ref{f3} implies that $F(l)=c \ind$. \\
We have
\begin{center}
    $\Vert F(g_d)-F(l)\Vert = \vert F(g_d)(\sigma(a_d))-F(l)(\sigma(a_d))\vert \longrightarrow 0$. 
\end{center}
Hence, $F(g_d)\longrightarrow F(l)$ and the continuity of $F^{-1}$ implies that $g_d \longrightarrow l $, which is a contradiction since $\Vert g_d -l\Vert = \vert g_d(a_d)-l(a_d)\vert = a-c $ for every $d$.

\medskip Let us turn to Fact~\ref{l12}. Keeping the notation of the proof we have given for it assuming the metrizability of $K$, one can modify the definition of the function $l$ by setting $l(y)=\mathrm{sgn}(h_i)(h+\frac{v}2)$ if $y\in V_i$ and $y\neq x$. In this way, the function $\chi$ does not appear any more and hence the metrizability assumption is no longer used. The definition of $l$ shows that $\vert l(y)\vert \leqslant u_{0,l}$ for all $y\in K\setminus(K'\cup\{ x\})$; and since $\Phi^{-1}$ is continuous (because $F^{-1}$ is continuous), we can apply Corollary~\ref{l11} and the remark following it to conclude that $\Phi^{-1}(l)(y)=l(y)$ for all $y\in K\setminus(K'\cup\{ x\})$. Once this is known, the remaining of the proof works without any change.

\smallskip
\section{Additional results}\label{S5}
In this section, we prove the following Theorem.

\begin{theorem}\label{main3} Let $K$ be a zero-dimensional compact Hausdorff space with a dense set of isolated points such that $K'$ is a $G_\delta$ subset of $K$, and let $F:B_{C(K)}\to B_{C(K)}$ be a non-expansive bijection. Under any of the following additional assumptions, one can conclude that $F$ is an isometry.
\begin{itemize}
\item[\rm (i)] $F$ has the following property~$(*)$: for any $u,v,z\in K$ with $u\neq v$ and for any sign $\omega\in\{ -1,1\}$, one can find $f\in B_{C(K)}$ such that $f(u)f(v)$ is non-zero with sign $\omega$ and $\vert F(f)\vert \equiv 1$ in a neighbourhood of $z$.
\item[\rm (ii)]  $\inf\,\{ \Vert F(\omega \ind_{\{a\}})\Vert;\; a\in K \minus K', \; \omega = \pm 1\} >1/2$.
\end{itemize}
\end{theorem}
\begin{proof}
\smallskip Looking back at the proof of Theorem~\ref{main}, we see that in fact, we only need to check one thing, namely that Lemma~\ref{l1} holds true as stated if either $F$ has property~$(*)$, or $\inf\,\{ \Vert F(\omega \ind_{\{a\}})\Vert;\; a\in K \minus K', \; \omega = \pm 1\} >1/2$. Indeed: once Lemma~\ref{l1} is known to be true, the assumption that $K'$ is finite is no longer used in the proof of Theorem~\ref{main}; under the assumption that $K'$ is $G_\delta$, Fact~\ref{l12} holds true as stated; and for the other parts of the proof of Theorem~\ref{main}, no additional assumption is needed on either $K'$ or $F$.

\smallskip\noindent The only trouble with the proof we have given for Lemma~\ref{l1} is at the end of the proof of Step~4, where we cannot conclude immediately that the map $\sigma$ is injective. So we have to find another way of showing that $\sigma$ is indeed injective, without assuming that $K'$ is  a finite set. 

\smallskip\noindent Before showing that $\sigma$ is injective, let us state the following fact: for every $a \in K \minus K'$, we have 
\begin{equation}\label{eqbiz} {\rm sgn}\Bigl(F(\omega \ind_{\{ a\}}) (\sigma(a))\Bigr)=\omega\, \alpha(a) \qquad\hbox{for $\omega =\pm 1$}.
\end{equation}
Indeed, we have shown at the end of the proof of Fact~\ref{keyformula} that for every $a \in K\minus K'$, we have $\alpha(a)={\rm sgn}\bigl(F(\ind_{\{ a\}}) (\sigma(a))\bigr)$. This implies (\ref{eqbiz}) since $F(\ind_{\{ a\}}) (\sigma(a))$ and $F(-\ind_{\{ a\}}) (\sigma(a))$ have opposite signs by Step~1 of the proof of Lemma~\ref{l1}.

\smallskip\noindent Towards a contradiction, assume that $\sigma$ is not injective, so one can find $u\neq v$ in $K$ such that $\sigma(u)=\sigma(v):= z$.\\
Since $\sigma(K') \subset K'$ and $\sigma_0$ is injective, we have $u,v \in K'$, so let us fix two nets $(a_i), (b_j)$ in $K\setminus K'$ such that $a_i\longrightarrow u$ and $b_j\longrightarrow v$. 

\smallskip\noindent
- Assume that $F$ satisfies property~$(*)$. Then, one can find $f\in B$ such that $\alpha(u)f(u)<0<\alpha(v)f(v)$ and $\vert F(f)\vert\equiv 1$ in a neighbourhood of $z$. By continuity of $\alpha$ and $f$, we have $\alpha(a_i) f\bigl(a_i)<0<\alpha\bigl(b_j)  f(b_j)$ for all large enough $i,j$, so that  $\Vert f+ \alpha(a_i) \mathds{1}_{\{ a_i\}}\Vert\leqslant 1$ and $\Vert f- \alpha( b_j) \mathds{1}_{\{b_j\}}\Vert\leqslant 1$.  Since $F$ is non-expansive, this implies in particular that $\vert F(f)(\sigma(a_i))- F(-\alpha(a_i)\mathds{1}_{\{a_i\}})(\sigma(a_i))\vert \leqslant 1$ and $\vert F(f)(\sigma(b_j))- F(\alpha(b_j)\mathds{1}_{\{b_j\}})(\sigma(b_j))\vert \leqslant 1$ for all large enough $i,j$. However, since $\sigma(a_i)\longrightarrow z$ and $\sigma(b_j)\longrightarrow z$ (because $\sigma$ is continuous), we also have $\vert F(f)(\sigma(a_i))\vert=1=\vert F(f)(\sigma(b_j))\vert$ for all large enough $i,j$. Since we know by (\ref{eqbiz}) that $F(-\alpha(a_i)\mathds{1}_{\{a_i\}})(\sigma(a_i))<0<F(\alpha(b_j)\mathds{1}_{\{b_j\}})(\sigma(b_j))$, it follows that $F(f)(\sigma(a_i))=-1$ and $F(f)(\sigma(b_j))=1$ for all large enough $i,j$. Hence $F(f)(z)$ is both equal to $-1$ and $1$, which is the required contradiction.

\smallskip\noindent
- Assume now that $c:=\inf\,\bigl\{ \Vert F(\omega \ind_{\{a\}})\Vert;\; a\in K \minus K', \; \omega = \pm 1\} >1/2$. Choose a function $f\in B$ with $\Vert f\Vert=1/2$ such that $f\cdot \alpha\equiv-1/2$ in a neighbourhood of $u$ and $f \cdot \alpha \equiv 1/2$ in a neighbourhood of $v$. Then, $f(a_i)\alpha(a_i)=-1/2$ and $f(b_j)\alpha(b_j)=1/2$ for all large enough $i,j$; which implies that $\Vert f+\alpha(a_i)\mathds{1}_{\{ a_i\}}\Vert= 1/2$ and $\Vert f-\alpha(b_j)\mathds{1}_{\{ b_j\}}\Vert= 1/2$. So we get $\vert F(f)(\sigma(a_i))-F(-\alpha(a_i)\mathds{1}_{\{a_i\}})(\sigma(a_i))\vert \leqslant 1/2$ and $\vert F(f)(\sigma(b_j))- F(\alpha(b_j)\mathds{1}_{\{b_j\}})(\sigma(b_j))\vert \leqslant 1/2$ for all large enough $i,j$. But $F(-\alpha(a_i)\mathds{1}_{\{a_i\}})(\sigma(a_i))\leqslant -c$ and $F(\alpha(b_j)\mathds{1}_{\{b_j\}})(\sigma(b_j))\geqslant c$, by (\ref{eqbiz}) and the definition of $c$. So we see that $F(f)(\sigma(a_i))\leqslant -c+1/2$ and $F(f)(\sigma(b_j))\geqslant c-1/2$ for all large enough $i,j$. Hence, $F(f)(z)\leqslant -c+1/2<0<c-1/2\leqslant F(f)(z)$, which is again a contradiction.
\end{proof}
\begin{corollary} If $K$ is a countable compact Hausdorff space, then any non-expansive bijection $F:B_{C(K)}\to B_{C(K)}$ either sending extreme points to extreme points or such that $\inf\,\{ \Vert F(f)\Vert;\; \Vert f\Vert=1\} >1/2$ is an isometry.
\end{corollary}
\begin{proof} The space $K$ is zero-dimensional with a dense set of isolated points, and it is also metrizable (so that the closed set $K'$ is $G_\delta$). Hence, the result follows from Theorem~\ref{main3}, (i) and (ii).
\end{proof}

\begin{corollary} Any non-expansive bijection $F:B_{\ell_\infty}\to B_{\ell_\infty}$ either sending extreme points to extreme points or such that $\inf\,\{ \Vert F(f)\Vert;\; \Vert f\Vert=1\} >1/2$ is an isometry.
\end{corollary}
\begin{proof} The space $\ell_\infty$ is isometric to $C(\beta\N)$, where $\beta\N$ is the Stone-$\check{\rm C}$ech compactification of $\N$. The space $K=\beta\N$ is zero-dimensional with a dense set of isolated points, and $K'=\beta\N\setminus\N$ is a $G_\delta$ subset of $\beta\N$. So we may apply Theorem~\ref{main3}, (i) and (ii).
\end{proof}

\end{document}